\documentclass[11pt]{amsart}

\usepackage{amsmath, amsthm, amssymb}

\usepackage[dvipsnames]{xcolor}
\definecolor{darkgreen}{RGB}{0,150,0}
\definecolor{darkred}{RGB}{200,0,0}
\definecolor{darkblue}{RGB}{0,50,205}
\definecolor{darkpurple}{RGB}{100,15,200}
\definecolor{gold}{RGB}{255,218,9}
\usepackage[colorlinks, linkcolor=darkblue, citecolor=darkgreen, urlcolor=darkgreen]{hyperref}
\hypersetup{hypertexnames=false}

\usepackage{newpxtext}
\usepackage[bigdelims,vvarbb]{newpxmath}  
\usepackage[T1]{fontenc}  

\usepackage{setspace}
\usepackage[margin=1in]{geometry}
             
\usepackage{comment}
\usepackage{thmtools, thm-restate}
\usepackage{overpic}
\usepackage{subcaption}

\usepackage{aliascnt}
\usepackage[noabbrev,capitalize,nameinlink]{cleveref}

\newtheorem{theorem}{Theorem}[section]
\newtheorem{proposition}[theorem]{Proposition}
\newtheorem{lemma}[theorem]{Lemma}

\newtheorem{conjecture}[theorem]{Conjecture}
\theoremstyle{definition}
\newtheorem{definition}[theorem]{Definition}
\newtheorem{example}[theorem]{Example}

\theoremstyle{remark}
\newtheorem{remark}[theorem]{Remark}
\numberwithin{equation}{section}

\usepackage{mleftright}
\newcommand{\set}[2]{\mleft\{\,#1~\middle|~#2\,\mright\}}
\newcommand{\R}{\mathbb{R}}
\newcommand{\C}{\mathbb{C}}

\newcommand{\N}{\mathbb{N}}
\newcommand{\Z}{\mathbb{Z}}

\newcommand{\ve}{\varepsilon}

\newcommand{\tb}{\mathrm{tb}}

\newcommand{\RNum}[1]{\uppercase\expandafter{\romannumeral #1\relax}}

\makeatletter 
 \def\l@subsection{\@tocline{2}{0pt}{2pc}{6pc}{}} \makeatother

\begin{document}

\title{Derivative links in contact topology}

\author{Joseph Breen}
\address{University of Alabama, Tuscaloosa, AL}
\email{jjbreen@ua.edu} \urladdr{https://sites.google.com/view/joseph-breen}

\author{Alexander Zupan}
\address{University of Nebraska - Lincoln, Lincoln, NE}
\email{zupan@unl.edu} \urladdr{http://www.math.unl.edu/~azupan2}

\thanks{JB was partially supported by NSF grant DMS-2038103 and an AMS-Simons Travel Grant. AZ was partially supported by NSF grant DMS-2405301 and a Simons Travel Support award.}

\begin{abstract}
We import the theory of $R$-links and derivative links into contact topology in both the Legendrian and transverse setting. This framework is used to characterize various forms of Lagrangian and symplectic sliceness, and more generally establishes one approach to what we call the Lagrangian slice-ribbon conjecture. Our main theorem asserts that a Legendrian knot with Thurston-Bennequin invariant $-1$ is Lagrangrian slice (resp.\ regularly Lagrangian slice) if and only if it supports a tight transverse (resp.\ Legendrian) $R$-link derivative.
\end{abstract}

\maketitle

\tableofcontents

\section{Introduction}

The slice-ribbon conjecture is one of the most compelling open problems in low-dimensional topology.  In this article, we adapt the approaches in \cite{miller2023handle} from the smooth setting to the setting of contact and symplectic topology in order to better understand potential gradations between sliceness and ribbonness in this context.  In particular, we explore Lagrangian and symplectic sliceness of knots in $(S^3, \xi_{\mathrm{st}})$. The strategy is to initiate a study in contact topology of $R$-links (\cref{subsec:tightRlinks}) and derivative links (\cref{subsec:derivativechar}), which we expect to be of independent interest beyond our own applications. Our main result, involving language that will be defined and contextualized below, is a characterization of Lagrangian sliceness in terms of $R$-links and derivative links.

\begin{theorem}\label{thm:main}
A Legendrian knot $\Lambda\subset (S^3, \xi_{\mathrm{st}})$ with $\mathrm{tb}(\Lambda) = -1$ is Lagrangian slice (resp.\ regularly Lagrangian slice) if and only if it supports a tight transverse (resp.\ Legendrian) $R$-link derivative.
\end{theorem}

\noindent In a forthcoming follow-up paper \cite{BZ1forthcoming}, we use \cref{thm:main} to establish a case of a conjecture of Cornwell, Ng, and Sivek \cite[Conjecture 4.1]{cornwell2016obstructions} on characterizing Lagrangian sliceness, and a case of a well-known question featured in the K3 problem list \cite[Problem 1.70]{K3} on sharpness in the slice-Bennequin inequality.

\subsection{Context}

The motivation for this work is a Lagrangian version of the slice-ribbon conjecture, which we discuss briefly before expanding on our main results.

A knot $K\subset S^3$ is \emph{slice} if it bounds a properly and smoothly embedded \emph{slice disk} $D^2 \subset B^4$, and is \emph{ribbon} if it bounds an immersed disk $D^2 \looparrowright S^3$ with only ribbon singularities. Ribbon knots are slice, and Fox's slice-ribbon conjecture \cite{fox1962knot} asking if the converse holds has become one of the most famous open questions in knot theory; see \cite{gilmer1982slicegenus,CG3,hass1983sliceribbon,miyazaki,lisca2007ribbon,gompf2010property2R,greene2011pretzel,lecuona2012montesinos,abe2016fibered,baker2016fibered,zemke2019ribbonKFH,aceto2021pretzel,dai2024figureeight}.

There are intermediaries between ribbonness and sliceness of interest. For instance, a knot is \emph{handle-ribbon} if it bounds a disk in $B^4$ whose exterior is built from $4$-dimensional $0$-, $1$-, and $2$-handles. As ribbon implies handle-ribbon \cite{gordon1981ribbon}, a refinement of the slice-ribbon conjecture is summarized by the chain of implications
\begin{equation}\label{eq:SR}
\{\text{ribbon}\} \Rightarrow \{\text{handle-ribbon}\} \Rightarrow \{\text{slice}\}.  
\end{equation}
The validity of each reverse implication is open \cite{miller2023handle}.

In contact and symplectic topology, a rich interaction of knots and surfaces manifests in the study of exact Lagrangian surfaces in $(B^4, \omega_{\mathrm{st}})$ bounded by Legendrian knots in $(S^3, \xi_{\mathrm{st}})$ \cite{eliashberg1996local,chantraine2010concordance,hayden2015fillable,cornwell2016obstructions,ekholm2016cobordisms,shende2019cluster,conway2021disks,casals2022infinitely}. Analogous to the refined slice-ribbon conjecture in \eqref{eq:SR}, there are gradations of Lagrangian sliceness, which we briefly explain below:
\begin{equation}\label{eq:SSR}
\{\text{decomposably slice}\} \Rightarrow \{\text{regularly slice}\} \Rightarrow \{\text{Lagrangian slice}\}.     
\end{equation}
A \emph{decomposable} Lagrangian disk is one built from certain local surgery moves on knot projections which induce Lagrangian $0$-handle and $1$-handle attachments \cite{ekholm2016cobordisms}. This has historically been the primary method for constructing exact Lagrangian surfaces and is a reasonable Lagrangian analogue of ribbonness. Informally, a \emph{regular} Lagrangian disk in $(B^4, \omega_{\mathrm{st}})$ is one whose exterior admits a Weinstein structure \cite{eliashberg2020flexible}, i.e.\ a certain $4$-dimensional symplectic $0$-, $1$-, $2$-handle decomposition. The parallel to handle-ribbonness is clear.

Decomposable Lagrangian disks are regular \cite{conway2021disks}, but we do not know if the converse holds. Likewise, we do not know if every Lagrangian disk is regular \cite{eliashberg2020flexible}. In fact, under any interpretation (for prescribed disks, for prescribed Legendrian knots, or simply for topological knot types) each reverse implication in the "Lagrangian slice-ribbon conjecture" \eqref{eq:SSR} is open. See \cite[\S 2]{breen2024regularlysliceimpliesoncestably} for an additional survey and comparison between \eqref{eq:SR} and \eqref{eq:SSR}.

Just as the class of smoothly slice knots can be further widened (e.g.\ to topologically or algebraically slice knots), we highlight a widening of Lagrangian sliceness for topological knot types. A transverse knot in $(S^3, \xi_{\mathrm{st}})$ is \emph{symplectically slice} if it bounds a symplectic disk in $(B^4, \omega_{\mathrm{st}})$, and a topological knot type $K$ is \emph{symplectically slice} if it admits a symplectically slice transverse representative. Work of Eliashberg \cite{eliashberg1995pushoff} implies that if a knot type $K$ is Lagrangian slice, then it is symplectically slice. This implication is not reversible: for example, $m(8_{20})$ is symplectically slice, but does not admit a Legendrian representative which is Lagrangian slice.

\subsection{Tight \(R\)-links}\label{subsec:tightRlinks}

The first element of our approach to studying sliceness in contact topology is the introduction of $R$-links into the Legendrian and transverse setting.

Recall that an \emph{$R$-link}, named after Gabai's property $R$ theorem \cite{gabai1987propertyR}, is a $g$-component link along which some integral surgery produces $\#^g (S^1 \times S^2)$. Gabai had shown that the only $R$-knot is the $0$-framed unknot. The generalized property $R$ conjecture \cite{gompf2010property2R} asks whether every $R$-link is handleslide equivalent to a $0$-framed unlink, veracity of which would prove that handle-ribbon implies ribbon. This conjecture is wide open, as are various weaker "stabilized" versions \cite{meier2016classification}. 

In this article, we import $R$-links into contact topology in two ways. First, recall that $\#^g (S^1 \times S^2)$ has a unique tight (in fact, Stein fillable) contact structure. 

\begin{definition}\label{def:tightlegRlink}
Let $L = L_1 \cup \cdots \cup L_g\subset (S^3, \xi_{\mathrm{st}})$ be a Legendrian link. We say that $L$ is a \emph{tight Legendrian $R$-link} if contact-$(+1)$ surgery along $L$ produces tight $\#^g (S^1 \times S^2)$.
\end{definition}

\begin{definition}\label{def:tighttranRlink}
Let $T = T_1 \cup \cdots \cup T_g\subset (S^3, \xi_{\mathrm{st}})$ be a transverse link. We say that $T$ is a \emph{tight transverse $R$-link} if inadmissible transverse $0$ surgery along $T$ produces tight $\#^g (S^1 \times S^2)$.
\end{definition}

\noindent 
We will often abbreviate the language to \emph{Legendrian $R$-link} and \emph{transverse $R$-link}, with the adjective "tight" left implicit. On the other hand, we say \emph{tight $R$-link} to broadly refer to both types. 

Contact and transverse surgery will be reviewed in \cref{sec:background}. For now, we recall that contact surgery coefficients are measured relative to the framing induced by the contact structure; the topological Dehn surgery coefficient in \cref{def:tightlegRlink} is necessarily $0$.

Max-tb Legendrian unlinks are Legendrian $R$-links \cite{Ding2009HandleMI}, and max-sl transverse unlinks are transverse $R$-links \cite{lisca2011transverseinvariants,conway2019transverse}. We can moreover rephrase known results in the literature \cite[Theorem 1.7, Theorem 1.8(ii)]{casals2024steintrace} in terms of property $R$: the max-tb Legendrian unknot is the only Legendrian $R$-knot, and the max-sl transverse unknot is the only transverse $R$-knot.

With stable versions of generalized property $R$ in mind, we call two Legendrian $R$-links $L, L'$ \emph{stably equivalent} if there are split max-tb unlinks $U$ and $U'$ such that $L \sqcup U$ is contact-$(+1)$ handleslide equivalent to $L'\sqcup U'$. Appealing to the uniqueness of Weinstein fillings of $(S^3, \xi_{\mathrm{st}})$, we prove: 

\begin{theorem}[Stable Legendrian Generalized Property $R$]\label{thm:stable_leg_gen_prop_r}
Every Legendrian $R$-link is stably equivalent to a max-tb unlink. 
\end{theorem}

\noindent Consequently, Legendrian $R$-links satisfy the stable generalized property $R$ conjecture 
\cite{meier2016classification}. 

There are many questions about tight $R$-links for which we do not know the answer. We record some examples here: 

\begin{restatable}{question}{tightRlink}\label{q:tightRlink}
Let $L = L_1 \cup \cdots \cup L_g \subset (S^3, \xi_{\mathrm{st}})$ be a Legendrian link with $g\geq 2$ such that 
\begin{enumerate}
    \item $L$ is a topological $R$-link, and 
    \item $\mathrm{tb}(L_i) = -1$ for each $i=1, \dots, g$.
\end{enumerate}
Is $L$ a Legendrian $R$-link? That is, is contact-$(+1)$ surgery along $L$ necessarily tight?   
\end{restatable}

\begin{restatable}[Legendrian Generalized Property $R$]{question}{genLegPropR}\label{q:genLegPropR}
Is every Legendrian $R$-link contact-$(+1)$ handleslide equivalent to a max-tb unlink?  
\end{restatable}

\begin{restatable}[Stable Transverse Generalized Property $R$]{question}{genTranPropR}\label{q:genTranPropR}
Is every transverse $R$-link stably equivalent to a max-sl unlink?  
\end{restatable}

\subsection{Derivative characterizations of contact sliceness}\label{subsec:derivativechar}

The second element of our approach is to import the notion of derivative links into the contact setting. Given a knot $K= \partial F\subset S^3$ bounding an orientable Seifert surface, recall that a \emph{derivative}, or \emph{metabolizer}, of $K$ is a link $J\subset F$ such that $F - J$ is connected and planar and the Seifert form of $F$ vanishes on the summand of $H_1(F; \Z)$ generated by $J$. Existence of such a link is necessary for sliceness, and the study of sliceness via derivatives has a rich history \cite{levine1969knotcobordismcodim2,CG1,gilmer1983sliceknots,CG2,cochran2010derivatives,gilmer2013surgerycurves,cochran2015davis,meier2022gensquare,miller2023handle}. We highlight the following derivative characterizations of sliceness in \eqref{eq:SR}, in increasing order of difficulty: 

\begin{proposition}[\cite{cochran2015davis,miller2023handle}]\label{lemma:ribbon_derivative}
A knot is ribbon if and only if it admits a Seifert surface with an unlink derivative.    
\end{proposition}

\begin{theorem}[\cite{miller2023handle}]\label{thm:miller_zupan}
A knot is handle-ribbon if and only if it admits a Seifert surface with an $R$-link derivative.
\end{theorem}

\begin{conjecture}[Stable Kauffman conjecture \cite{cochran2015davis}]\label{conj:stabKauff}
A knot is slice if and only if it admits a Seifert surface with a slice link derivative. 
\end{conjecture}

\noindent Note that \cref{conj:stabKauff} is a generalization of the well-known conjecture that a knot is smoothly slice if and only if its untwisted Whitehead double is smoothly slice \cite[Problem 1.38]{Kirby1995ProblemsIL}. Indeed, the standard genus $1$ Seifert surface of $\mathrm{Wh}_0(K)$ admits a derivative isotopic to $K$ (by viewing the original companion knot $K$ as the "core" of the untwisted band). If $K$ is slice, then $\mathrm{Wh}_0(K)$ is slice. \cite[Problem 1.38]{Kirby1995ProblemsIL} (and the "only if" direction of \cref{conj:stabKauff}) asks if the converse holds.

Our first derivative characterization is a Weinstein analogue of \cref{thm:miller_zupan}.

\begin{restatable}{theorem}{legKauffman}\label{thm:legKauffman}
Let $\Lambda \subset (S^3, \xi_{\mathrm{st}})$ be a Legendrian knot with $\mathrm{tb}(\Lambda) = -1$. Then $\Lambda$ is regularly Lagrangian slice if and only if it admits a Seifert surface with a Legendrian $R$-link derivative.
\end{restatable}

\noindent Regarding symplectic sliceness of transverse knots, we obtain: 

\begin{restatable}{theorem}{tranKauffman}\label{thm:tranKauffman}
Let $\mathcal{T} \subset (S^3, \xi_{\mathrm{st}})$ be a transverse knot with $\mathrm{sl}(\mathcal{T}) = -1$. The following are equivalent.
\begin{enumerate}
    \item The knot $\mathcal{T}$ is symplectically slice.\label{part:TK1}
    \item The knot $\mathcal{T}$ admits a Seifert surface with a transverse $R$-link derivative.\label{part:TK2} 
    \item The knot $\mathcal{T}$ admits a Seifert surface with a max-sl unlink derivative.\label{part:TK3}
\end{enumerate}
\end{restatable}

\noindent By work of Rudolph \cite{rudolph1983seifertribbons} and Boileau and Orevkov \cite{boileau2001quasipositive}, knot types with transverse representatives bounding symplectic surfaces are precisely those that are closures of quasipositive braids. Mark and Tosun \cite[Corollary 2.10]{mark2024fillable} proved that if a smoothly slice Legendrian knot $\Lambda$ with $\mathrm{tb}(\Lambda) = -1$ has a quasipositive transverse pushoff, then $\Lambda$ is Lagrangian slice. This, together with \cref{thm:legKauffman} and \cref{thm:tranKauffman}, implies the statement of \cref{thm:main}.

With \cref{lemma:ribbon_derivative} and the analogy between \eqref{eq:SR} and \eqref{eq:SSR} in mind, we also prove that decomposably slice knots support max-tb unlink derivatives: 

\begin{restatable}{theorem}{decKauffman}\label{thm:decKauffman}
Let $\Lambda \subset (S^3, \xi_{\mathrm{st}})$ be a Legendrian knot. If $\Lambda$ is decomposably Lagrangian slice, then it admits a Seifert surface with a max-tb unlink derivative. 
\end{restatable}

\noindent For subtle reasons that will be discussed later in the article (see \cref{remark:lSGPRC} and \cref{sec:bandsum}), the converse direction (a max-tb unlink derivative implies decomposable sliceness) is more difficult and we content ourselves to leave it as a conjecture. 

\begin{conjecture}\label{conj:decomposable}
Let $\Lambda \subset (S^3, \xi_{\mathrm{st}})$ be a Legendrian knot with $\mathrm{tb}(\Lambda) = -1$. If $\Lambda$ admits a Seifert surface with a max-tb unlink derivative, then $\Lambda$ is decomposably Lagrangian slice.
\end{conjecture}

\subsection{Organization} In \cref{sec:background} we collect background information on Legendrian and transverse knots, convex surface theory, exact Lagrangian cobordisms, contact surgery, and transverse surgery. In \cref{sec:leg_r_link} we introduce Legendrian $R$-links and prove \cref{thm:stable_leg_gen_prop_r}. Using a local model established in \cref{sec:bandsum}, we prove the derivative results \cref{thm:legKauffman} and \cref{thm:decKauffman} in \cref{sec:reg_char}. Finally, in \cref{sec:transverse} we repeat the story in the transverse setting, proving \cref{thm:tranKauffman} (hence \cref{thm:main}).

\subsection{Acknowledgments} The authors would like to thank John Etnyre, James Hughes, and Josh Sabloff for interest in our work, Kyle Hayden for a useful correspondence, and Austin Christian and B\"ulent Tosun for several discussions and for offering feedback on an early draft. 

\subsection{Statement on AI use} Large language models were not used in the preparation of this article: the ideas and text were conceived and written by the authors.

\section{Background}\label{sec:background}

\subsection{Legendrian and transverse knots}

Let $\xi_{\mathrm{st}}$ denote the unique tight (positive) contact structure on $S^3$, which is the (maximally non-integrable) plane field of complex tangencies to $S^3\subset \C^2$. There are two natural classes of knots in $(S^3, \xi_{\mathrm{st}})$: \emph{Legendrian} knots and \emph{transverse} knots. The former are everywhere tangent to $\xi_{\mathrm{st}}$, while the latter are everywhere (positively) transverse. 

Beyond their topological type, Legendrian knots have two additional classical invariants. The first is the \emph{Thurston-Bennequin number}, denoted $\mathrm{tb}(\cdot)$, which is the integral framing induced by the contact structure. The second is the \emph{rotation number}, denoted $\mathrm{rot}(\cdot)$, which is the relative Euler class of the contact structure on a Seifert surface, changing sign according to a choice of orientation of the knot. Transverse knots only have one classical invariant. The \emph{self-linking number} of a transverse knot $\mathcal{T}\subset (S^3, \xi_{\mathrm{st}})$ defined as follows. Choose a Seifert surface $\Sigma$ and a nonzero section $s: \Sigma \to \xi\mid_{\Sigma}$, which is possible as the vector bundle $\xi\mid_{\Sigma}$ is trivial. Then $\mathrm{sl}(\mathcal{T}):= \mathrm{lk}(\mathcal{T}, s(\mathcal{T}))$. Most importantly for us, if $\mathcal{T}$ is the positive transverse pushoff of a Legendrian knot $\Lambda$, then $\mathrm{sl}(\mathcal{T}) = \mathrm{tb}(\Lambda) - \mathrm{rot}(\Lambda)$.

Given an oriented Legendrian knot $\Lambda \subset (S^3,\xi_{\mathrm{st}})$, we denote by $S_{\pm}(\Lambda)$ the \emph{positive} (resp.\ \emph{negative}) stabilization as in \cref{fig:stabs}. Note that $\mathrm{tb}(S_{\pm}(\Lambda)) = \mathrm{tb}(\Lambda) - 1$ and $\mathrm{rot}(S_{\pm}(\Lambda)) = \mathrm{rot}(\Lambda) \pm 1$. In particular, if $\mathcal{T}, \mathcal{T}_-$ denote the positive transverse pushoffs of $\Lambda, S_-(\Lambda)$, respectively, then $\mathrm{sl}(\mathcal{T}) = \mathrm{sl}(\mathcal{T}_-)$. In fact, more is true, as $\mathcal{T}$ and $\mathcal{T}'$ are always transversely isotopic.

\begin{figure}[ht]
	\centering
	\begin{overpic}[scale=.26]{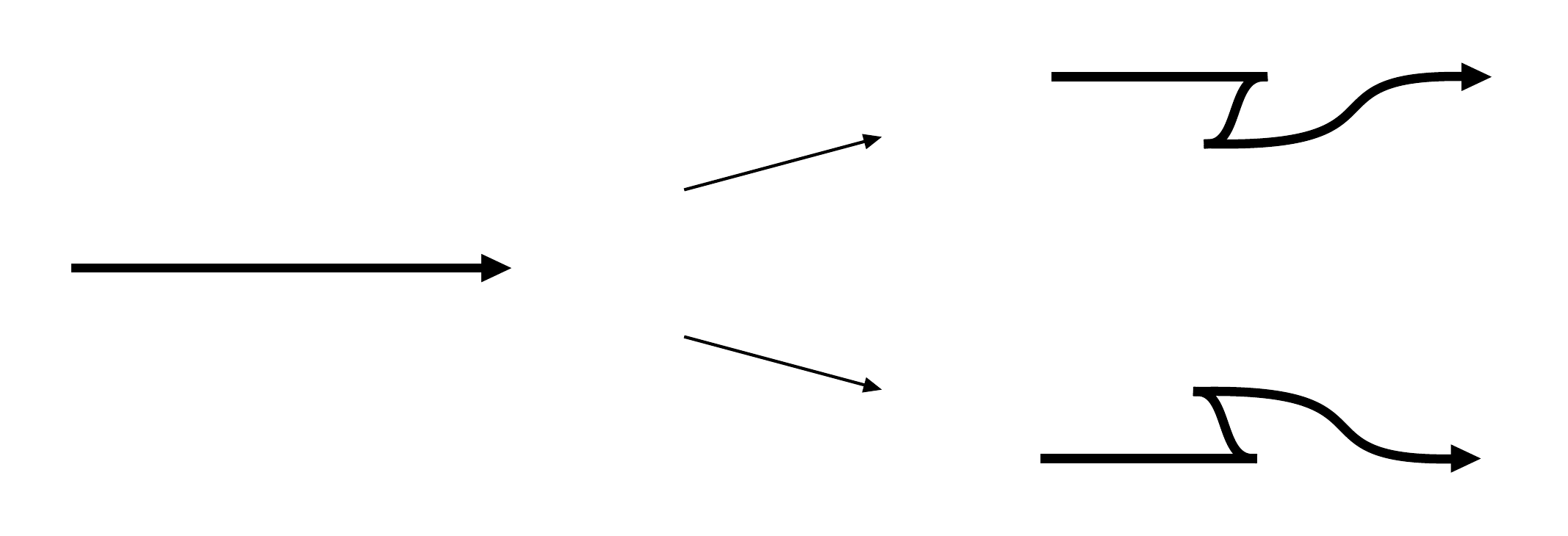}
       \put(48,26){\small $S_+$}
       \put(48,7){\small $S_-$}
	\end{overpic}
	\caption{Positive and negative Legendrian stabilization in the front projection.}
	\label{fig:stabs}
\end{figure}

\subsection{Convex surface theory} Here we collect the relevant background on convexity, deferring to the literature references such as \cite{giroux1991convexite,honda2000classification} for more details, along with the book of Geiges \cite{geiges2008introduction} and the lecture notes of Etnyre \cite{etnyre2004convexlectures} and Massot \cite{massot2013topologicalmethods3dimensionalcontact}. 

Let $(Y, \xi)$ be a co-oriented contact $3$-manifold. Recall that a \emph{contact vector field} $X$ on $Y$ is one satisfying $\mathcal{L}_X\xi = \xi$, i.e.\ one whose flow preserves the contact distribution. 

\begin{definition}
An oriented surface $\Sigma\subset (Y, \xi)$, compact and possibly with Legendrian boundary, is \textit{convex} if there is a contact vector field $X$ everywhere transverse to $\Sigma$.    
\end{definition}

\begin{definition}
Let $\Sigma$ be a convex surface. Given a contact vector field $X$ witnessing convexity, the \emph{dividing set} is $\Gamma_X := \set{p\in \Sigma}{X_p \in \xi_{p}}$. The \emph{positive} (resp.\ \emph{negative}) region, denoted $R_{\pm}(\Sigma)$, is the set of points $p\in \Sigma$ where $X_p$ is positively (resp.\ negatively) transverse to $\xi_p$.   
\end{definition}

Informally, the dividing set indicates where the contact structure is "vertically" transverse to the surface, where verticality is measured by $X$. It turns out that $\Gamma_X$ is always a smoothly and properly embedded multicurve, and its isotopy class in $\Sigma$ is independent of the choice of $X$. For this reason, we often refer to the dividing set as $\Gamma$ without reference to a specific contact vector field. 

The dividing set encodes the contact structure in a neighborhood of a convex surface. Specifically, if $\xi$ and $\xi'$ are two contact structures defined near $\Sigma$ inducing the same dividing set $\Gamma_X$ with respect to a fixed contact vector field $X$, then $\xi$ and $\xi'$ are contact isotopic near $\Sigma$ relative to $\Gamma_X$. We refer to this principle as Giroux flexibility \cite[Theorem 3.4]{honda2000classification}. Consequently, in any neighborhood of a convex surface there is an infinitely-large vertically-invariant neighborhood determined up to contact isotopy by the isotopy class of the dividing curves, obtained by integrating the transverse contact vector field:

\begin{lemma}[\cite{giroux1991convexite}]\label{lemma:nbd_size}
Let $\Sigma \subset (Y, \xi)$ be a convex surface. Let $\mathcal{U} \subset Y$ be any open neighborhood of $\Sigma$. There is a contact embedding $(\Sigma \times \R_t, \ker \alpha_{\mathrm{inv}})\hookrightarrow (\mathcal{U}, \xi)$ mapping $\Sigma\times \{0\}$ to $\Sigma \subset \mathcal{U}$, such that 
\begin{enumerate}
    \item $\partial_t$ is a strict contact vector field, so that $\Sigma\times \{0\} \subset \Sigma \times \R$ is convex and $\alpha_{\mathrm{inv}}$ is $\R$-invariant, and 
    \item the dividing set of $\Sigma \times \{0\}$ is mapped to the dividing set of $\Sigma\subset \mathcal{U}$. 
\end{enumerate}
Specifically, we may take $\alpha_{\mathrm{inv}}$ to be of the form $\alpha_{\mathrm{inv}} = f\, dt + \beta$ where $f, \beta$ are $t$-invariant, the dividing set is $\Gamma =\{f=0\}$, and $R_{\pm}(\Sigma) = \{\pm f > 0\}$.
\end{lemma}

Convex surface theory interacts well with Legendrian knots. For instance, Seifert surfaces with Legendrian boundary can be made convex, provided $\mathrm{tb}$ is non-positive.

\begin{theorem}[\cite{giroux1991convexite,honda2000classification}]\label{thm:legendrian_bdry_convex}
Let $\Sigma\subset (Y, \xi)$ be an orientable surface with Legendrian boundary $\partial \Sigma = \Lambda_1 \cup \cdots \cup \Lambda_n$. Assume that $\mathrm{tb}(\Lambda_i)\leq f_i$, where $f_i$ is the smooth framing induced by $\Sigma$. Then there is a perturbation of $\Sigma$ relative to $\partial \Sigma$, $C^0$-small near $\partial \Sigma$ and $C^{\infty}$-small in the interior, such that the resulting surface is convex. Moreover, $\mathrm{tb}(\Lambda_i) = f_i -\frac{1}{2}|\Gamma \cap \Lambda_i|$ and $\mathrm{rot}_{[\Sigma]}(\partial \Sigma) = \chi(R_+(\Sigma)) - \chi(R_-(\Sigma))$. 
\end{theorem}

Two convex surfaces meeting along a Legendrian corner, convex with respect to different vector fields, admit a standard neighborhood along which the edge can be rounded \cite[Lemma 3.11]{honda2000classification}. Given interlacing dividing curves, the edge rounding principle is "up and to the left;" see \cref{fig:edge}. 

\begin{figure}[ht]
	\centering
	\begin{overpic}[scale=.37]{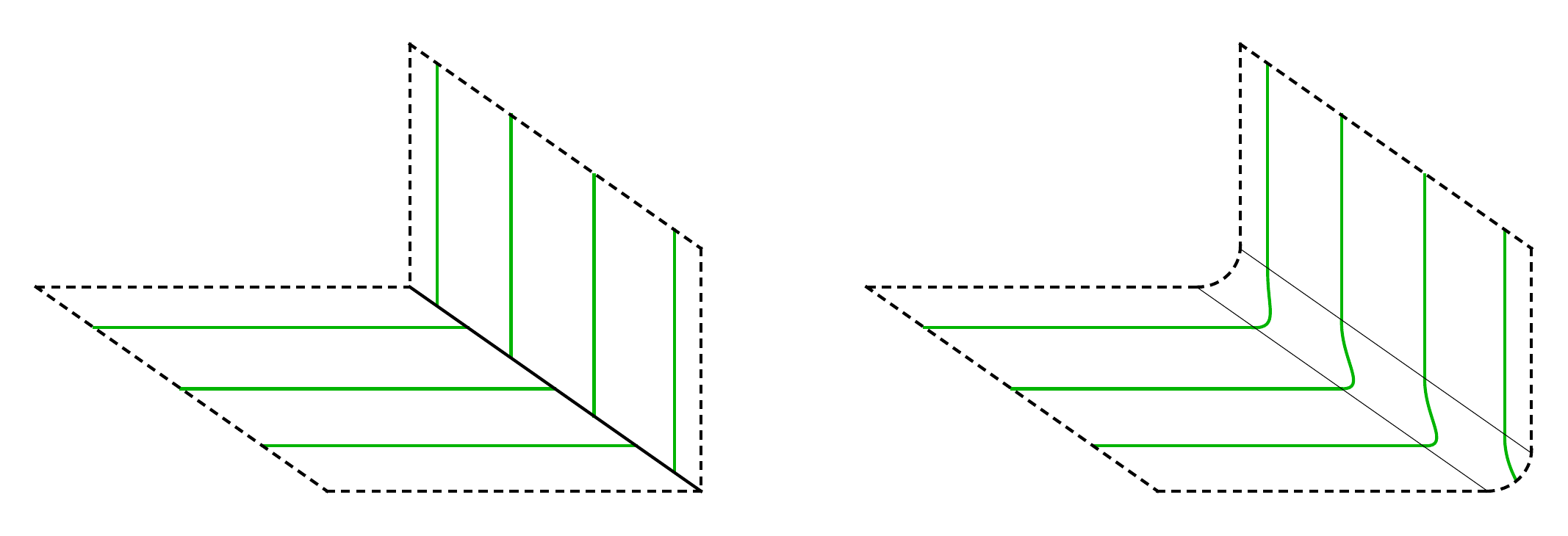}
       
	\end{overpic}
	\caption{Edge-rounding two convex surfaces along a Legendrian corner.}
	\label{fig:edge}
\end{figure}

Convexity also provides a framework to realize embedded curves and arcs on surfaces as Legendrian. If $\Sigma$ is convex with dividing set $\Gamma$, we say that a properly embedded graph $C\subset \Sigma$ is \textit{non-isolating with respect to $\Gamma$} if each component of $\Sigma - C$ has nonempty intersection with $\Gamma$. Informally, non-isolating curves on convex surfaces may be assumed to be Legendrian.

\begin{theorem}[Legendrian realization principle \cite{kanda1998legendrian,honda2000classification}]\label{thm:LeRP}
Let $\Sigma \subset (Y, \xi)$ be a convex surface with dividing set $\Gamma$. Let $C\subset \Sigma$ be a properly embedded graph transverse to and non-isolating with respect to $\Gamma$. There is an isotopy from the identity $\phi_s:\Sigma \to \mathcal{U}\subset Y$, $s\in [0,1]$, supported in the invariant neighborhood of $\Sigma$, such that
\begin{enumerate}
    \item $\phi_s(\Sigma)$ is graphical (in particular convex) in the invariant neighborhood, and 
    \item $\phi_1(C)$ is Legendrian.
\end{enumerate}
Moreover, wherever $C$ is already Legendrian, the isotopy may be taken to be constant. 
\end{theorem}

Dividing curves on convex surfaces also allow us to detect tightness of neighborhoods. A contact structure is \emph{overtwisted} if there is a Legendrian unknot with $\mathrm{tb} = 0$, and is \emph{tight} otherwise. Overtwisted contact structures are classified by their underlying algebro-topological data \cite{eliashberg1989class_OT} and thus tight structures are geometrically richer. Important to us is the fact that $S^3$ and $\#^g(S^1 \times S^2)$ have unique tight contact structures, each induced as the boundary of a Weinstein domain.

\begin{theorem}[Giroux criterion \cite{honda2000classification}]\label{thm:criterion}
If $\Sigma \neq S^2$ is a convex surface in a contact manifold $(Y, \xi)$, then $\Sigma$ has
a tight neighborhood if and only if the dividing set $\Gamma$ has no homotopically trivial curves. When $\Sigma = S^2$, $\Sigma$ has a tight neighborhood if and only if $\Gamma$ has a single connected component.
\end{theorem}

Finally, we discuss the concept of a bypass attachment as introduced by Honda \cite{honda2000classification}. Given an isotopy $\Sigma_s$, $s\in [0,1]$, of a surface in a contact manifold, convexity generically holds for all but finitely many $s$. Before and after these momentary failures of convexity, the dividing curves on the surface may evolve. Bypass attachments describe the fundamental unit of change. 

A \emph{bypass half-disk} for a surface $\Sigma$ is a topological disk $D$ with piecewise-smooth Legendrian boundary $\partial D = a \cup b$, where $a\subset \Sigma$ is an embedded Legendrian arc transversally intersecting the dividing set thrice, twice at its endpoints, $b$ is properly embedded in the complement of $\Sigma$, the natural smoothing of $a\cup b$ is a $\mathrm{tb}=-1$ unknot, and $D$ is convex with dividing curve given by the purple arc in the middle of \cref{fig:bypassfigure}. Given the existence of a bypass half-disk $D$ for a surface $\Sigma$, one may isotope $\Sigma$ past $D$ to obtain a new convex surface $\Sigma'$ with dividing set $\Gamma'$ as indicated in the lower right part of \cref{fig:bypassfigure}; see \cite[Lemma 3.12]{honda2000classification}. We call the process of isotoping $\Sigma$ past $D$ to obtain $\Sigma'$ \emph{attaching a bypass to $\Sigma$}.

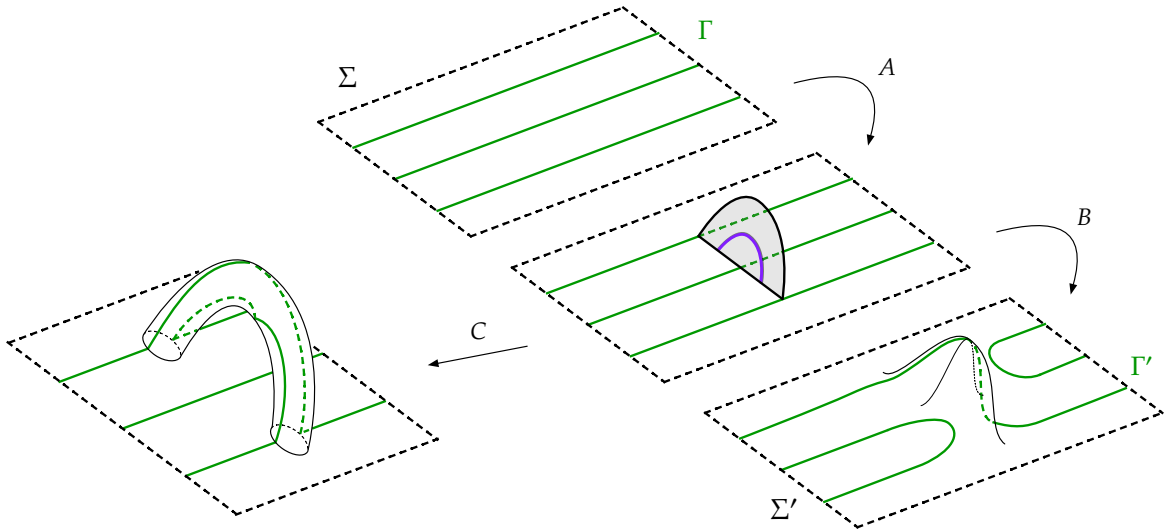
\begin{figure}[ht]
	\centering
	\begin{overpic}[scale=.44]{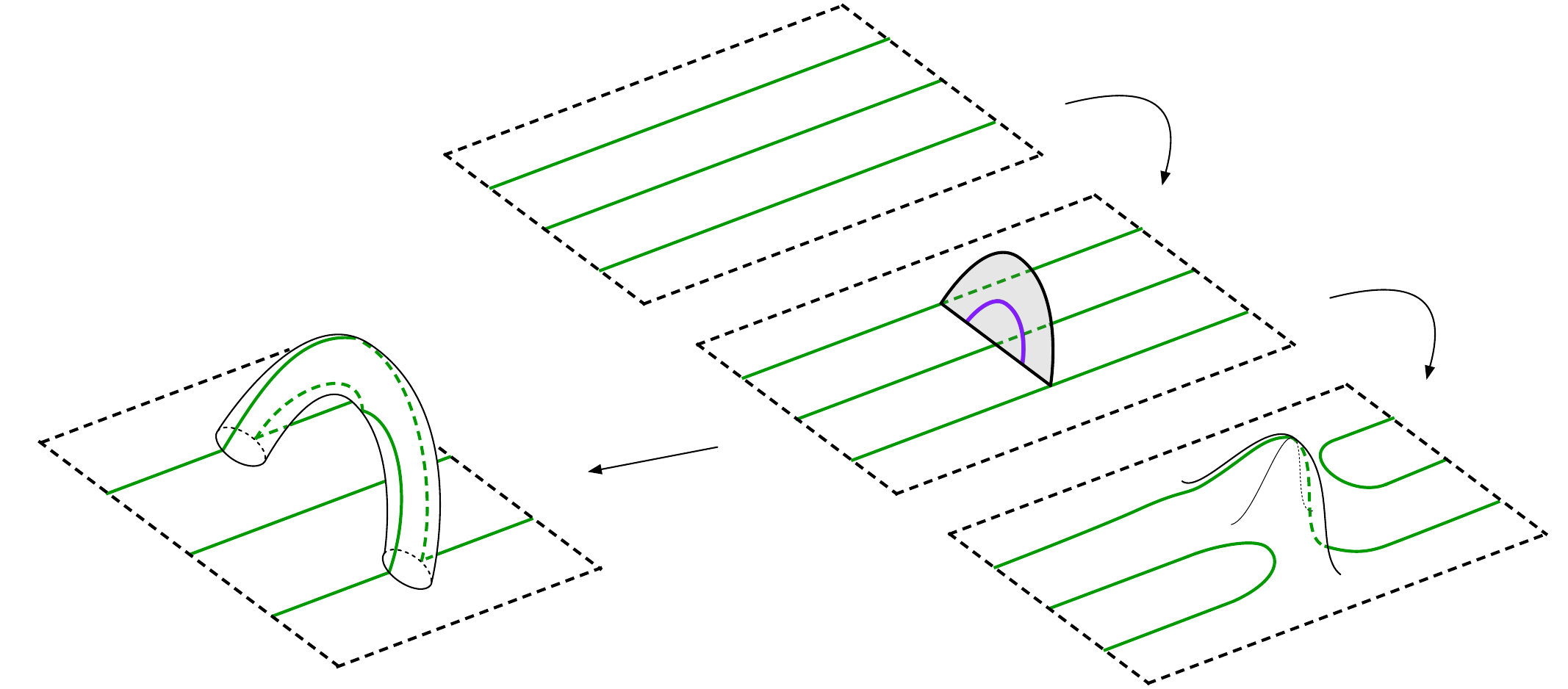}
       \put(30,37){$\Sigma$}
       \put(60,41){\small \textcolor{darkgreen}{$\Gamma$}}
       \put(66,1){$\Sigma'$}
       \put(96,13){\small \textcolor{darkgreen}{$\Gamma'$}}
        \put(75,38){\footnotesize $A$}
        \put(91.5,25.5){\footnotesize $B$}
        \put(41,16){\footnotesize $C$}
	\end{overpic}
	\caption{Bypass attachment. Arrow $A$ indicates the identification of a bypass half disk, shaded in gray, with Legendrian boundary and purple dividing set. Arrow $B$ indicates the result of isotoping $\Sigma$ past the bypass half disk, i.e.\ attaching the bypass to $\Sigma$, to obtain $\Sigma'$ with dividing curves $\Gamma'$. Arrow $C$ indicates the result of only attaching the associated contact $1$-handle.}
	\label{fig:bypassfigure}
\end{figure}

A different perspective on bypass attachments was recorded by Ozbagci \cite{ozbagci2011contact}. The neighborhood of a bypass disk --- roughly, the region bounded by $\Sigma$ and $\Sigma'$ in \cref{fig:bypassfigure} --- is given by a smoothly canceling pair of handles of index $1$ and $2$. These handles have specific models relative to the ambient contact structure and are thus called \emph{contact handles}. Roughly, the arc $b$ is the core of the $1$-handle, which is a standard jet bundle neighborhood of $b$, and (a slight retraction of) the bypass-half disk is the core of the $2$-handle. The effect of attaching only the contact-$1$ handle to $\Sigma$ is described by the lower-left part of \cref{fig:bypassfigure}.

\subsection{Exact Lagrangian cobordisms}

We are interested in the following cobordism structure between Legendrian knots. 

\begin{definition}
Let $\Lambda_-, \Lambda_+ \subset (Y^3, \xi=\ker\alpha)$ be oriented Legendrian links in a co-oriented contact manifold. A \emph{Lagrangian cobordism from $\Lambda_-$ to $\Lambda_+$} is an embedded oriented Lagrangian surface $L\subset (\R_s \times Y, d\lambda_{\mathrm{st}})$, where $\lambda_{\mathrm{st}}:=e^s\, \alpha$, such that 
\begin{enumerate}
    \item $L \cap ([-s_0, s_0]\times Y)$ is compact for any $s_0\in \R_{\geq 0}$, and 
    \item there exists a $C>0$ such that 
    \begin{align*}
        L \cap ([C, \infty) \times Y) &= [C, \infty) \times \Lambda_+, \\
        L \cap ((-\infty, -C] \times Y) &= (-\infty, -C] \times \Lambda_-.
    \end{align*}
\end{enumerate}
If $\lambda_{\mathrm{st}}\mid_L = df$ for some function $f:L \to \R$ which is constant on $[C, \infty) \times \Lambda_+$ and $(-\infty, -C] \times \Lambda_-$, then $L$ is \emph{exact}.
\end{definition}

\subsubsection{Regular Lagrangian cobordisms}

In the symplectic setting, the natural analogue of a smooth handle decomposition is a Weinstein handle decomposition. We review definitions briefly and refer to the standard reference for more details \cite{cieliebak2012stein}. A \emph{Weinstein structure} on a $4$-dimensional compact cobordism $(W, \partial_- W, \partial_+W)$ is a tuple $(\lambda, \phi)$ where
\begin{enumerate}
    \item $d\lambda$ is symplectic and the \emph{Liouville vector field} $X_{\lambda}$ defined by $\iota_{X_{\lambda}}d\lambda = \lambda$ is outwardly (resp.\ inwardly) transverse to $\partial_+ W$ (resp.\ $\partial_- W$), 
    \item $\phi: W \to \R$ is Morse with $\partial_{\pm} W = \phi^{-1}(c_{\pm})$ a regular level set, and 
    \item for some Riemannian metric, the Liouville vector field $X_{\lambda}$ is gradient-like for $\phi$. 
\end{enumerate}
If $\partial_- W = \emptyset$, we write $(W, \lambda, \phi)$ and call the cobordism a \emph{Weinstein domain}. A \emph{Weinstein homotopy} is a $1$-parameter family $(W, \partial_- W, \partial_+ W, \lambda_t, \phi_t)$ of Weinstein structures on a fixed cobordism $W$, allowing for embryonic critical points. Two Weinstein cobordisms are \emph{deformation equivalent} if they are Weinstein homotopic under the pullback by a diffeomorphism. 

Every Weinstein cobordism $(W, \partial_- W, \partial_+ W, \lambda, \phi)$ can be canonically \emph{completed} by attaching cylindrical ends
\[
 ((-\infty,0]_s \times \partial_- W, \, e^s\, \lambda\mid_{\partial_- W}) \, \cup \, (W, \partial_- W, \partial_+ W, \lambda) \, \cup \, ([0,\infty)_s \times \partial_+ W, \, e^s\, \lambda\mid_{\partial_+ W})
\]
and extending $\phi$ by an appropriate affine shift of the $s$-coordinate. Weinstein homotopic cobordisms have symplectomorphic completions \cite[Corollary 11.21]{cieliebak2012stein}, and we will usually not distinguish between a Weinstein cobordism and its completion.

\begin{remark}
The Liouville condition implies that critical points of $4$-dimensional Weinstein structures have index at most $2$.     
\end{remark}

\begin{remark}
A \emph{Stein structure} is a complex structure admitting an exhausting strictly plurisubharmonic function. Every Stein structure yields a Weinstein structure, unique up to Weinstein homotopy, and conversely every Weinstein structure arises from a Stein structure. That is, there is an isomorphism on $\pi_0$ of the space of Weinstein and Stein structures; conjecturally, the spaces are weak homotopy equivalent \cite{cieliebak2012stein}. For this reason, a casual reader may freely exchange the words "Weinstein" and "Stein" according to preference.
\end{remark}

This prelude on Weinstein topology is primarily for the purpose of introducing the notion of regularity for Lagrangian submanifolds, which plays the symplectic role of handle-ribbonness.

\begin{definition}[\cite{eliashberg2020flexible}]
A properly embedded Lagrangian cobordism $L\subset (W, \lambda, \phi)$ in a Weinstein cobordism is \emph{regular} if there is a Weinstein homotopy $(\lambda, \phi) \rightsquigarrow (\lambda',\phi')$ such that $L$ is Lagrangian throughout the homotopy and the Liouville vector field $X_{\lambda'}$ is tangent to $L$. 
\end{definition}

A consequence of the definition is that if $L$ is regular, the complement of a neighborhood of $L$ admits a Weinstein structure. Most pertinent is the consequence that if $D\subset (B^4, \lambda_{\mathrm{st}},\phi_{\mathrm{st}})$ is a regular Lagrangian disk in the standard radial Weinstein structure on the $4$-ball, then $B^4 - D$ admits the structure of a (completed) Weinstein domain.

\subsubsection{Decomposable Lagrangian cobordisms}

One explicit construction of exact Lagrangian cobordisms appeals to two local moves, which are \emph{isolated max-tb unknot birth} and \emph{ambient Legendrian surgery} as depicted in \cref{fig:decomposable}.

\begin{figure}[ht]
	\centering
	\begin{overpic}[scale=.25]{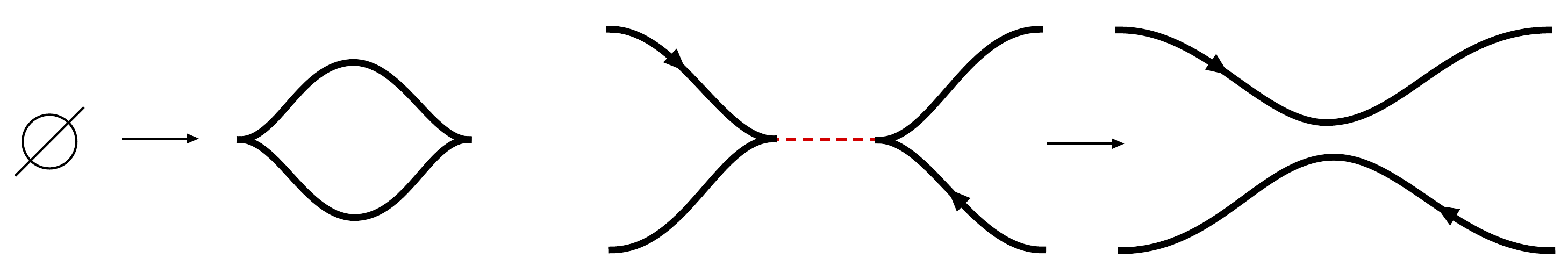}

	\end{overpic}
	\caption{The moves which define decomposable cobordisms: unknot birth, and (orientable) ambient Legendrian surgery.}
	\label{fig:decomposable}
\end{figure}

If $\Lambda_+\subset (Y, \xi)$ is obtained from $\Lambda_-$ by a sequence of Legendrian isotopies, isolated max-tb unknot births, and (orientable) ambient Legendrian surgeries, then there is an (orientable) exact Lagrangian cobordism in $\R \times Y$ from $\Lambda_-$ to $\Lambda_+$ \cite{ekholm2016cobordisms}. Any cobordism arising in this way is called \emph{decomposable}.

\subsection{Contact surgery}

Contact-$(r)$ surgery, introduced in full generality by Ding and Geiges \cite{ding2001symplectic}, see also \cite{DingGeigesStipsicz2004,Ding2009HandleMI}, is a contact Dehn surgery operation performed along a Legendrian knot $\Lambda$ which induces a topological $\mathrm{tb}(\Lambda) + r$ surgery. The contact surgery coefficient $(r)$ thus refers to the surgery framing relative to the contact planes. 

\begin{remark}
We distinguish topological surgery coefficients and contact surgery coefficients with parentheses. For example, if $\Lambda \subset S^3$ is a Legendrian knot, we write $S^3_{\Lambda}(r)$ for topological $r$-surgery along $\Lambda$ and $S^3_{\Lambda}((r))$ for contact-$(r)$ surgery, which induces topological $\mathrm{tb}(\Lambda) + r$ surgery, along $\Lambda$.   
\end{remark}

We are especially interested in contact surgeries that correspond to attachment and removal of symplectic handles as developed by Weinstein \cite{weinstein1991surgery}. A Weinstein $2$-handle is attached along a Legendrian knot and performs contact-$(-1)$ surgery. On the other hand, removal of a Weinstein $2$-handle, whose co-core is a (regular) Lagrangian disk filling a Legendrian belt-sphere, induces a contact-$(+1)$ surgery along the belt-sphere.

\begin{lemma}[\cite{DingGeigesStipsicz2004}]\label{lemma:surg_cancel}
Let $\Lambda \subset (Y, \xi)$ be a Legendrian knot, and let $\Lambda'$ be a Legendrian pushoff of $\Lambda$ in the Reeb direction. The contact manifold obtained by performing contact-$(-1)$ surgery along $\Lambda$ and contact-$(+1)$ surgery along $\Lambda'$ is contactomorphic to $(Y, \xi)$.  
\end{lemma}

Another aspect of contact surgery we will need is the description of contact-$(\pm1)$ handleslides \cite{Ding2009HandleMI}, specifically those arising as part of Weinstein homotopies. Weinstein homotopies of critical handles are executed by Legendrian isotopy of their attaching and belt spheres. Handleslides correspond to $\xi$-transverse intersections of Legendrian knots which occur in generic $1$-parameter families. As such, a handleslide is executed along a short nondegenerate Reeb chord between two attaching spheres or belt spheres. In the front projection, the model for sliding a Legendrian across a contact-$(\pm 1)$ surgery is given in \cref{fig:slidesplus1}. The left and right models are equivalent (differing only in the type of nondegeneracy of the Reeb chord), but it is convenient to appeal to both. 

\begin{figure}[ht]
	\centering
	\begin{overpic}[scale=.45]{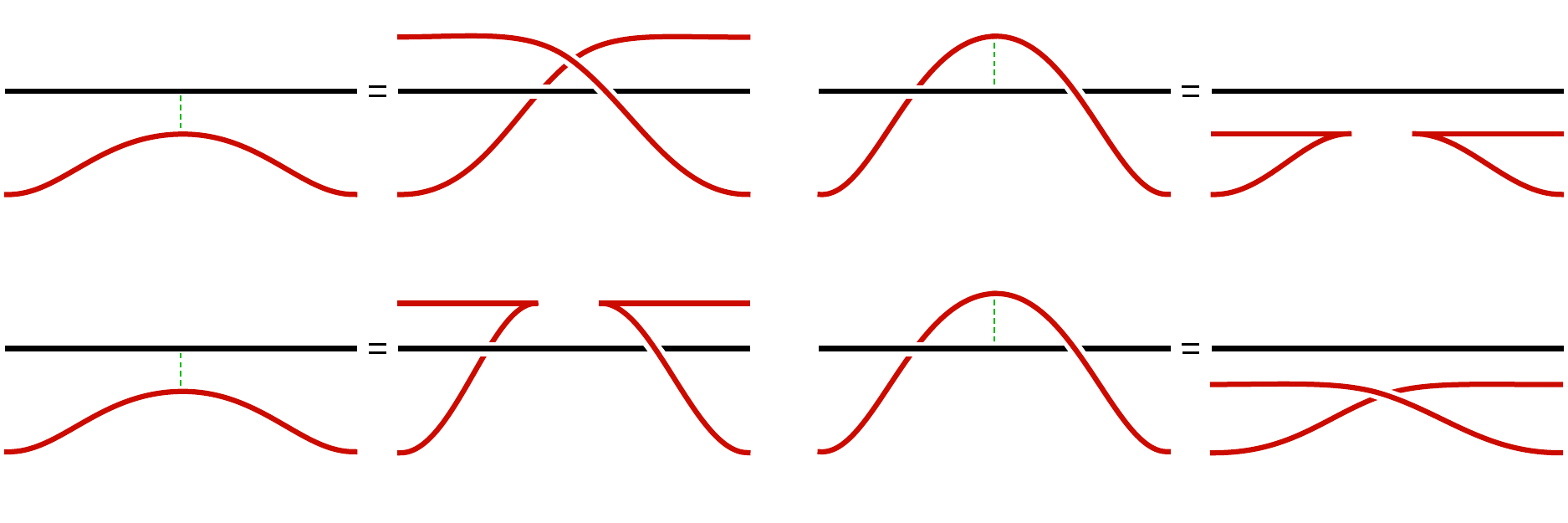}
        \put(18,28.5){\tiny $(+1)$}
        \put(43.25,28.5){\tiny $(+1)$}
        \put(70,28.5){\tiny $(+1)$}
        \put(95,28.5){\tiny $(+1)$}

        \put(18,12){\tiny $(-1)$}
        \put(43.25,12){\tiny $(-1)$}
        \put(70,12){\tiny $(-1)$}
        \put(95,12){\tiny $(-1)$}
       
	\end{overpic}
	\caption{Front projection models for a contact-$(\pm 1)$ handleslide.}
	\label{fig:slidesplus1}
\end{figure}
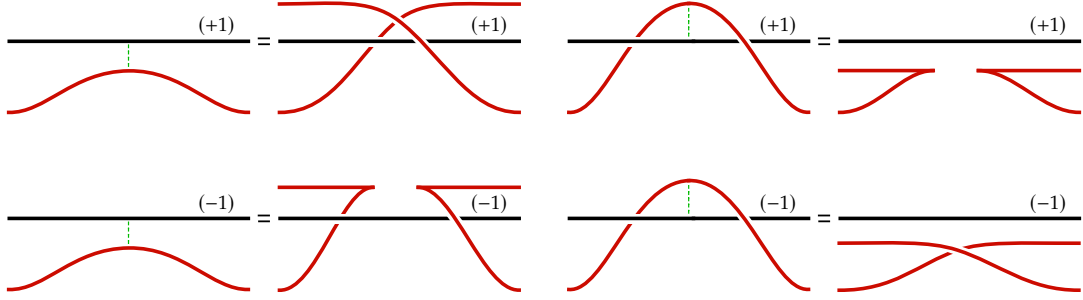

\subsubsection{Higher contact surgeries}

Next we briefly describe contact surgery with other coefficients besides $(\pm 1)$. For our purpose, it will suffice to consider contact-$(+n)$ surgery, where $n\in \N$. 

Given a Legendrian knot $\Lambda\subset (Y, \xi)$, excise a standard neighborhood to obtain a contact manifold $(Y_{\Lambda}, \xi)$ with convex torus boundary $\partial Y_{\Lambda} \cong T^2$ supporting two parallel dividing curves giving the contact framing of $\Lambda$. Choose a coordinate system on $T^2$ identifying the contact framing with slope $0$. To perform contact-$(+n)$ surgery we perform a topological slope-$n$ Dehn filling of $\Lambda\subset (Y, \xi)$ and then endow the surgery torus with a tight contact structure. 

By uniqueness of tight contact structures on the $3$-ball, the choice of contact structure on the surgery torus is determined by the contact structure in a neighborhood of the meridian disk. We may specify such a contact structure by specifying the isotopy class of the dividing curves on the disk. A slope $n$ meridian disk will intersect $\Gamma_{T^2}$ in $2n$-many spots, forcing $n$-many dividing arcs on the disk. These must all be boundary parallel, for otherwise, by edge-rounding of convex surfaces \cite[Lemma 3.11]{honda2000classification}, the boundary sphere obtained from gluing in the disk would have a homotopically trivial dividing curve and thus an overtwisted contact structure. Consequently, in the case $n>1$, there are two choices of contact structure, corresponding to the two different sets of boundary parallel arcs (differing in their cyclic parity).   

Given an orientation of $\Lambda$, there is a way to coherently assign a denotation of \emph{positive} or \emph{negative} to the two choices of contact structure; see, for example, \cite{lisca2011transverseinvariants}. The details are unimportant, and we will always use the negative choice of contact-$(+n)$ surgery, calling it the \emph{preferred choice}. The following property is sufficient for us.

\begin{proposition}[\cite{lisca2011transverseinvariants}]
Let $\Lambda\subset (Y, \xi)$ be an oriented Legendrian knot, and let $S_-(\Lambda)$ be a negative Legendrian stabilization. Let $n\in \N$. Then there is a contactomorphism
\[
Y_{\Lambda}((n)) \cong Y_{S_-(\Lambda)}((n+1))
\]
where both sides use the preferred contact surgery. 
\end{proposition}

\subsection{Transverse surgery}
Surgery along transverse knots originated in the work of Gay \cite{gay2002symplectichandles}, and was further developed and clarified by Baldwin and Etnyre \cite{baldwin2013admissible} and Conway \cite{conway2019transverse}. Transverse surgery comes in two flavors, admissible and inadmissible, which for our purposes should informally be thought of as corresponding to attaching and removing, respectively, a symplectic $2$-handle with symplectic core and co-core.  

We begin with admissible transverse surgery, which is more subtle. Let $\mathcal{T}'\subset (Y', \xi')$ be a transverse knot. There a neighborhood $N':=N(\mathcal{T}')$ of $\mathcal{T}'$ contactomorphic to $(S^1_z \times D^2_{r,\theta}(R), \ker(\cos r\, dz + r\sin r\, d\theta))$, where $R\in (0,\pi)$ and $D^2(R)$ is a disk of radius $R$. Relative to the $(z, \theta)$ coordinate system, the slope of the characteristic foliation on the torus $\{r = r_0\}$ is $-\tfrac{\cos r_0}{r_0 \sin r_0}$ (so that the slope tends to $-\infty$ as $r_0 \to 0^+$, and tends to $+\infty$ as $r_0 \to \pi^-$). With reference to this coordinate system, any topological framing $\lambda'$ of $\mathcal{T}'$ represented by a slope $s< -\tfrac{\cos R}{R \sin R}$ is called \emph{admissible}.

\begin{remark}
A leaf of the characteristic foliation on the torus $\{r = r_0\}$ gives a Legendrian approximation of $\mathcal{T}$. This is only well-defined up to negative Legendrian stabilization, as a torus with smaller radius (thus more negative slope) gives a Legendrian approximation related to the initial one by negative stabilization. We call any such approximation \emph{a Legendrian pushoff of $\mathcal{T}$}.
\end{remark}

The following theorem is then due to Gay \cite{gay2002symplectichandles}, with the "moreover" statement following from work of Wendl \cite[Theorem 1.5]{wendl2013symplectichandles}, but stated in the language of \cite[Theorem 2.2]{conway2019transverse}.

\begin{theorem}[\cite{gay2002symplectichandles,wendl2013symplectichandles}]\label{thm:admissible_handle}
Let $(X',\omega')$ be a symplectic manifold with strongly convex boundary, and let $\mathcal{T}' \subset \partial X'$ be a transverse knot with a chosen framing $\lambda'$. If $\lambda'$ is an admissible framing, then a $\lambda'$-framed symplectic $2$-handle may be attached along $\mathcal{T}'$ to produce a symplectic manifold $(X,\omega)$ with weakly convex boundary such that the contact structure on $\partial X$ is obtained from admissible transverse $\lambda'$ surgery along $\mathcal{T}'$. Moreover, the co-core of the handle attachment is a symplectic disk. 
\end{theorem}

To efficiently describe inadmissible transverse surgery, we appeal to work of Conway \cite{conway2019transverse}, who showed that every inadmissible transverse surgery can be performed by contact-$(r)$ surgery on a Legendrian approximation for some $r>0$. Specifying to integral surgery coefficients, let $\mathcal{T}\subset (Y, \xi)$ be a transverse knot and let $\lambda$ be a topological framing determining a $0$ slope as above. Let $\Lambda \subset (Y, \xi)$ be a Legendrian pushoff of $\mathcal{T}$. By further negative stabilization if necessary, we can assume that, in the coordinate system specified by $\lambda$, $\Lambda$ has slope $-n$ for some $n \in \N_{>0}$.

\begin{theorem}[\cite{conway2019transverse}]
Inadmissible transverse surgery along $\mathcal{T}$ with framing $\lambda$ is contact isotopic to the preferred contact-$(+n)$ surgery along $\Lambda$.
\end{theorem}

As far as identifying admissible surgeries, the following lemma will suffice. It is a consequence of the proof of \cite[Theorem 1.13]{conway2021disks}, but we isolate the argument for convenience. 

\begin{lemma}\label{lemma:admissible}
Let $\mathcal{T}\subset (Y, \xi)$ be a transverse knot and let $(Y', \xi')$ be the contact manifold obtained by an integral inadmissible transverse surgery along $\mathcal{T}$. Then there is a transverse dual knot $\mathcal{T}'\subset (Y', \xi')$ such that the dual surgery framing $\lambda'$ is admissible for $\mathcal{T}'$.
\end{lemma}

\begin{proof}
Let $\Lambda\subset (Y, \xi)$ be a Legendrian pushoff of $\mathcal{T}$ so that the inadmissible transverse surgery along $\mathcal{T}$ is witnessed by the preferred contact-$(+n)$ surgery along $\Lambda$ for some $n\in \N$. Let $N$ be a standard neighborhood of $\Lambda$, and choose a coordinate system on $\partial N$ so that the dividing curves have slope $0$. Then, to perform the contact-$(+n)$ surgery, we remove $N$ and glue in a preferred contact solid torus $N'$ with meridional slope $n$. Topologically, the core of this torus is the dual surgery knot, and the topological dual framing has slope $\infty$ (i.e.\ the slope of the meridian of $N$). It is known (see \cite[\S 2.3]{conway2021disks}, or the first paragraph of the proof of \cite[Theorem 1.13]{conway2021disks}) that in $N'$ one can find a transverse representative of the core together with a standard neighborhood witnessing any characteristic foliation slope counterclockwise of $0$ and clockwise of $n$ on the Farey graph. This implies that $\infty$ is an admissible slope.   
\end{proof}

\section{Legendrian \(R\)-links}\label{sec:leg_r_link}

In this section we study Legendrian $R$-links as in \cref{def:tightlegRlink} and prove \cref{thm:stable_leg_gen_prop_r}. For some preliminary observations, a Legendrian $R$-link is in particular a topological $R$-link. Consequently, $L$ is algebraically unlinked and the topological surgery framing is $0$ along each component \cite[Proposition 2.2]{gompf2010property2R}. It follows that each component $L_i$ of a Legendrian $R$-link has $\mathrm{tb}(L_i) = -1$. As we will see in the proof of \cref{thm:stable_leg_gen_prop_r}, any component of a Legendrian $R$-link is smoothly slice in $B^4$; thus, the slice-Bennequin inequality implies $\overline{\mathrm{tb}}(L_i)\leq -1$. The extent to which \cref{def:tightlegRlink} is restrictive beyond this framing condition is posed in \cref{q:tightRlink}.

The key examples of Legendrian $R$-links are as follows. 

\begin{example}[Max-tb unlinks are Legendrian $R$-links]\label{ex:max-tb_unlink_R_link}
Let $U = U_1 \cup \cdots \cup U_g$ be a $g$-component unlink of max-tb unknots. By \cite{Ding2009HandleMI}, contact-$(+1)$ surgery along $U$ yields tight $\#^g (S^1\times S^2)$, hence $U$ is a Legendrian $R$-link. 
\end{example}

\begin{example}[A non-trivial Legendrian $R$-link]
Consider the link at the top of \cref{fig:2Rlink}, which consists of a $m(9_{46})$ component and an unknot component. This is a Legendrian $R$-link, as can be seen by performing the indicated sequence of Legendrian isotopies and contact-$(+1)$ handleslides, which preserves the result of the contact surgery.    
\end{example}

\begin{figure}[ht]
	\centering
	\begin{overpic}[scale=.4]{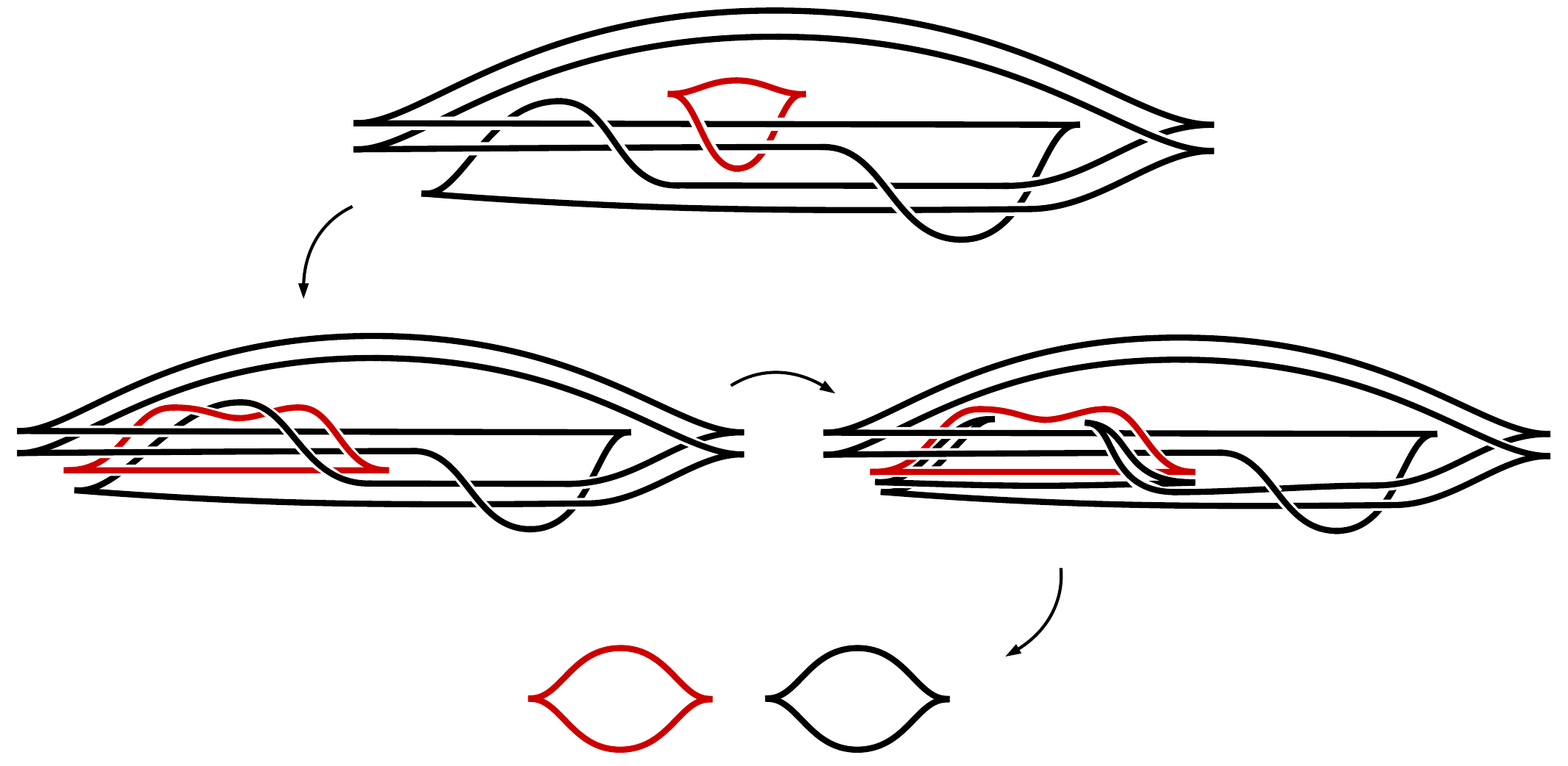}
        \put(21,32.5){\tiny isotopy}
        \put(60.5,10){\tiny isotopy}
        \put(40,26.5){\tiny contact-$(+1)$ handleslide}
	\end{overpic}
	\caption{A nontrivial tight Legendrian $R$-link.}
	\label{fig:2Rlink}
\end{figure}

Now we prove \cref{thm:stable_leg_gen_prop_r}, stating that Legendrian $R$-links are stably contact-$(+1)$ handleslide equivalent to max-tb unlinks. 

\begin{proof}[Proof of \cref{thm:stable_leg_gen_prop_r}.] 
A topological $R$-link gives rise to an "upside down" homotopy $4$-ball by attaching $0$-framed $2$-handles to $[0,\ve]\times S^3$ along the link in $\{\ve\}\times S^3$ and then capping off the resulting $\#^g(S^1\times S^2)$ boundary component with $\natural^g(S^1 \times B^3)$, viewing the latter portion as comprised of $3$-handles and a $4$-handle. We mimic the construction in the "right-side-up" Weinstein setting. 

Let $L\subset (S^3, \xi_{\mathrm{st}})$ be a $g$-component Legendrian $R$-link and denote
\[
(Y, \xi_{Y}) := S^3_{L}((+1)) \cong (\#^{g} (S^1 \times S^2), \,\xi_{\mathrm{st}}). 
\]
There is a Weinstein cobordism $W$ from $Y$ to $S^3$ given by attaching Weinstein $2$-handles along a small contact push-off of $L$ in $S^3_{L}((+1))$; indeed, the contact-$(-1)$ surgeries of the $2$-handles cancel the contact-$(+1)$ surgeries along $L$ by \cref{lemma:surg_cancel}. On the other hand, there is a unique Weinstein filling of $(Y, \xi_{Y})$ given by the standard Weinstein structure on $\natural^{g} (S^1 \times B^3)$ \cite[Theorem 16.9]{cieliebak2012stein}. Concatenating this filling of $Y$ with the cobordism $W$ gives a Weinstein filling of $(S^3, \xi_{\mathrm{st}})$ such that the Legendrian $R$-link is isotopic to the belt spheres of the $2$-handles of the filling.

Weinstein fillings of tight $S^3$ are unique up to deformation equivalence \cite{cieliebak2012stein}, and in particular are deformation equivalent to the standard radial structure on $B^4$. Thus, after a self-diffeomorphism of $B^4$ relative boundary (which does not affect the $R$-link in the boundary $S^3$) we may assume that the Weinstein filling $Y \cup W \cong B^4$ of tight $S^3$ is Weinstein homotopic (allowing for births and deaths of canceling $1$-/$2$-handle pairs) to the standard $0$-handle.

At the level of the initial $R$-link $L$, a $1$-/$2$-pair birth in the Weinstein homotopy is witnessed by adjoining an isolated max-tb unknot component $U$ to $L$. Indeed, this modifies $(Y, \xi_Y)$ to 
\[
(Y', \xi_{Y'}) := (Y, \xi_Y) \,\#\, (S^1 \times S^2, \xi_{\mathrm{st}}) \cong (\#^{g+1} (S^1 \times S^2),\, \xi_{\mathrm{st}})
\]
which is filled by attaching an additional $1$-handle to $\natural^{g} (S^1 \times B^3)$. Likewise, the cobordism $W'$ from $Y'$ to $S^3$ is built by including an additional $2$-handle along a small push-off of $U$, which cancels the $1$-handle corresponding to $U$. 

Now we consider Weinstein handleslides. As the correspondence between Weinstein structures and Legendrian $R$-links lies in the belt spheres of the $2$-handles, it suffices to consider the impact of $2$-handle slides on the corresponding $R$-link. To that end, consider the general situation of two Weinstein handles attached along Legendrian knots $\Lambda_1 \cup \Lambda_2 \subset (M, \xi)$ in a contact $3$-manifold, and let $L_1 \cup L_2 \subset M_{\Lambda_1 \cup \Lambda_2}((-1))$ denote the corresponding belt spheres. If $\Lambda_1' \cup \Lambda_2' \subset M$ is the result of sliding $\Lambda_1$ across the contact-$(-1)$ surgery along $\Lambda_2$, then the resulting belt spheres $L_1' \cup L_2'$ are obtained by sliding $L_2$ across the contact-$(+1)$ framed $L_1$. Moreover, the belt sphere of each handle is isotopic to max-tb unknot components linking once with the attaching sphere \cite[Proposition 2]{Ding2009HandleMI} as in the lower left of the figure. The resulting belt sphere contact-$(+1)$ handleslide equivalence is depicted by the sequence in the rest of the figure. 

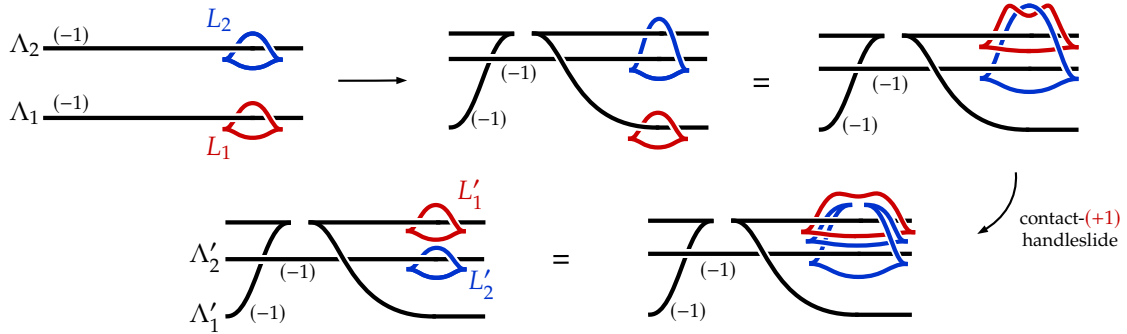
\begin{figure}[ht]
	\centering
	\begin{overpic}[scale=.4]{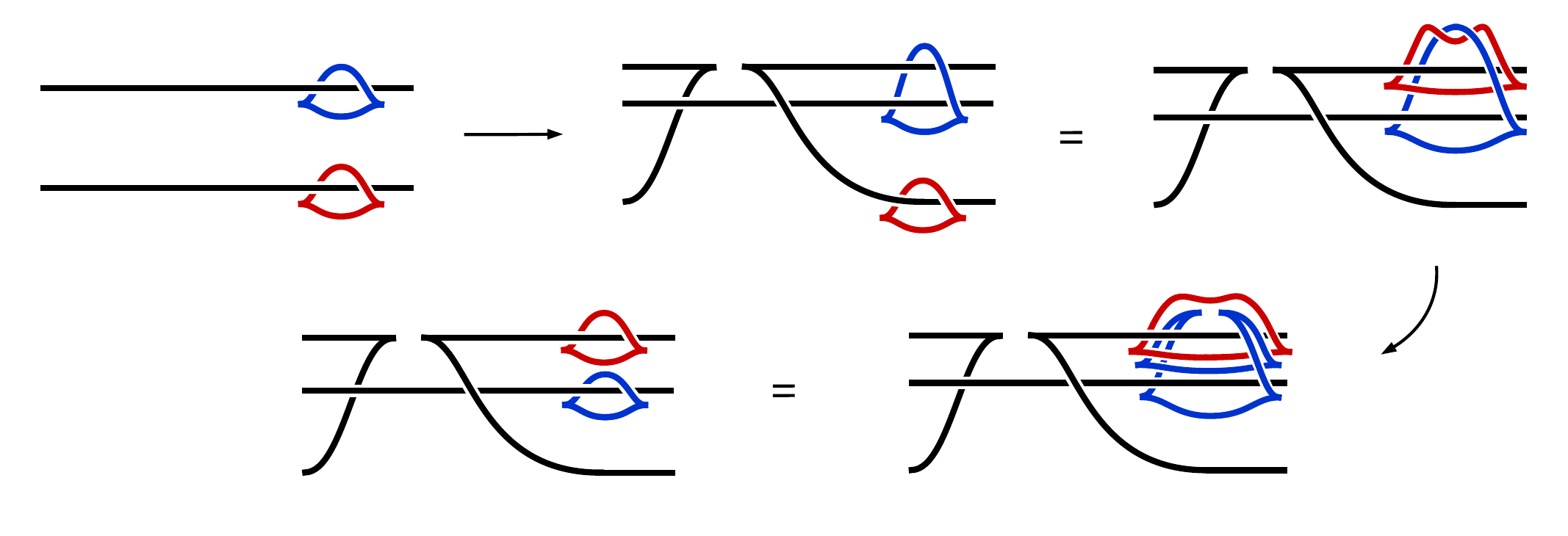}
\put(-0.5,22){\small $\Lambda_1$}
\put(-0.5,28.25){\small $\Lambda_2$}

\put(3.5,23){\tiny $(-1)$}
\put(3.5,29.25){\tiny $(-1)$}

\put(17.5,18.75){\small \textcolor{darkred}{$L_1$}}
\put(17.5,30.5){\small \textcolor{darkblue}{$L_2$}}

\put(16,3.5){\small $\Lambda_1'$}
\put(16,9){\small $\Lambda_2'$}

\put(41.75,21.5){\tiny $(-1)$}
\put(44.5,25.75){\tiny $(-1)$}

\put(75.75,21){\tiny $(-1)$}
\put(78.5,24.75){\tiny $(-1)$}

\put(60,4){\tiny $(-1)$}
\put(62.75,7.75){\tiny $(-1)$}

\put(21.5,3.75){\tiny $(-1)$}
\put(24.25,7.5){\tiny $(-1)$}

\put(40.5,14.75){\small \textcolor{darkred}{$L_1'$}}
\put(41.5,6.5){\small \textcolor{darkblue}{$L_2'$}}
        
        \put(92,12.5){\tiny contact-\textcolor{darkred}{$(+1)$}}
        \put(92.25,10.5){\tiny handleslide}
	\end{overpic}
	\caption{The handleslides in the proof of \cref{thm:stable_leg_gen_prop_r}.}
	\label{fig:slidesbelt}
\end{figure}

In such a Weinstein homotopy from our initial filling of $S^3$ to the radial filling, we may delay all $1$-/$2$-pair cancellations until the end, resulting in a geometrically canceling Weinstein structure. The corresponding Legendrian $R$-link is a max-tb unlink. By the above remarks on $1$-/$2$-pair birth and $2$-handleslides, the initial Legendrian $R$-link is stably equivalent to a max-tb unlink. 
\end{proof}

\begin{remark}[Ribbon, handle-ribbon, and stable generalized property $R$]\label{remark:SGPRC}
The stable generalized property $R$ Conjecture \cite{meier2016classification} (abbreviated SGPRC) posits that every $R$-link is stably equivalent to an unlink, where stable equivalence is accordingly defined by split unlink inclusion and $0$-framed handleslides. If SGPRC is true, then handle-ribbonness is equivalent to ribbonness: first, the unknot is ribbon, and second, $0$-framed handleslides of ribbon knots preserve ribbonness by viewing the result of a handleslide as a small band sum between the two ribbon knots.  Moreover, the SGPRC implies that every homotopy 4-ball with $S^3$ boundary and built without 3-handles is diffeomorphic to the standard 4-ball (which is equivalent to the smooth 4-dimensional Poincar\'e conjecture for geometrically simply-connected 4-manifolds).
\end{remark}

\begin{remark}[Decomposability, regularity, and stable Legendrian generalized property $R$]\label{remark:lSGPRC}
In contrast, it is not clear if the Legendrian version of SGPRC (\cref{thm:stable_leg_gen_prop_r}) implies equivalence of regular sliceness and decomposable sliceness. Specifically, while the max-tb unknot is decomposably slice, we do not know if contact-$(+1)$ handleslides preserve decomposable sliceness. The issue is in the type of band sum performed when doing a contact-$(+1)$ handleslide, which is not compatible with the usual formulation of decomposable cobordisms in the front projection. Additional discussion on these band sums is provided in the next section. Also in contrast to smooth setting, symplectic homotopy 4-balls with convex $S^3$ boundary are known to be standard by \cite{gromov1985pseudo}; in fact, this is a key component in our proof of \cref{thm:stable_leg_gen_prop_r}.
\end{remark}

\section{Legendrian band sums}\label{sec:bandsum}

Here we identify and discuss a key local model for Legendrian band sums (\cref{lemma:div_curve_h_slide}), which will be used in the subsequent section to prove \cref{thm:legKauffman}. It has been noted \cite{etnyre2001knotscontact,Ding2009HandleMI,casals2024steintrace} that there are multiple models in the front projection for performing a Legendrian band sum. For the connected sum of split links, the different models give Legendrian isotopic results \cite{Ding2009HandleMI}, but it is unclear to what extent different models are equivalent in general. The distinction is important when comparing decomposability and regularity, and we introduce some convenient referential language. 

\begin{figure}[ht]
	\centering
	\begin{overpic}[scale=.28]{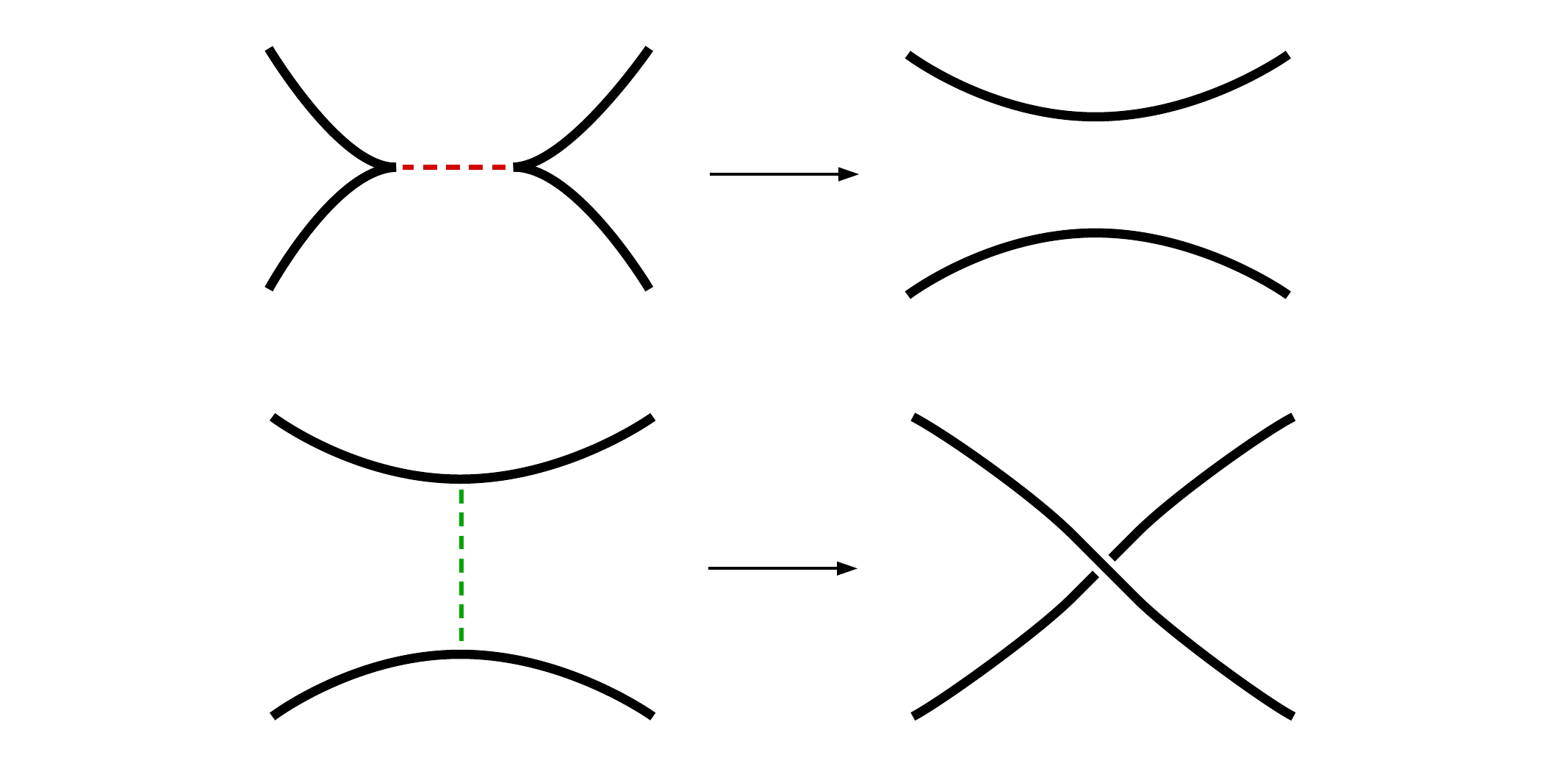}
       \put(40,14.5){\footnotesize Type-$R$ band}
       \put(41,39.5){\footnotesize Type-$D$ band}
	\end{overpic}
	\caption{Type-$D$ vs. Type-$R$ bands in the front projection.}
	\label{fig:typeDR}
\end{figure}

We call the top and bottom front projection models in \cref{fig:typeDR} band attachments, or moves, of \emph{Type-$D$} and \emph{Type-$R$}, respectively. A Type-$D$ band has Legendrian core, creates a new contractible Reeb chord, and induces a decomposable (hence the nomenclature) Lagrangian $1$-handle cobordism. On the other hand, a Type-$R$ band is one with a Reeb chord as core. The name is chosen to be reminiscent of "Reeb," "regular," and "$R$-link," settings in which Type-$R$ bands naturally appear in the present article.

\begin{proposition}
Every Type-$D$ band is locally induced by a Type-$R$ band. However, not every Type-$R$ band can be recovered by a sequence of Legendrian isotopies and a Type-$D$ band.     
\end{proposition}

\begin{proof}
The proof that every Type-$D$ move can be executed with a Type-$R$ is move is contained in \cref{fig:DfromR}. To show that the converse does not hold, consider the sequence of diagrams in \cref{fig:nonorientable}. Assume for the sake of contradiction that the Type-$R$ move in the middle row can be replaced by a sequence of Legendrian isotopies and and a Type-$D$ move. The modified sequence would then construct a non-orientable decomposable filling of the max-tb unknot. However, the max-tb unknot does not admit any non-orientable exact Lagrangian fillings \cite{chen2024nonorientable}. 
\end{proof}

\begin{figure}[ht]
	\centering
	\begin{overpic}[scale=.35]{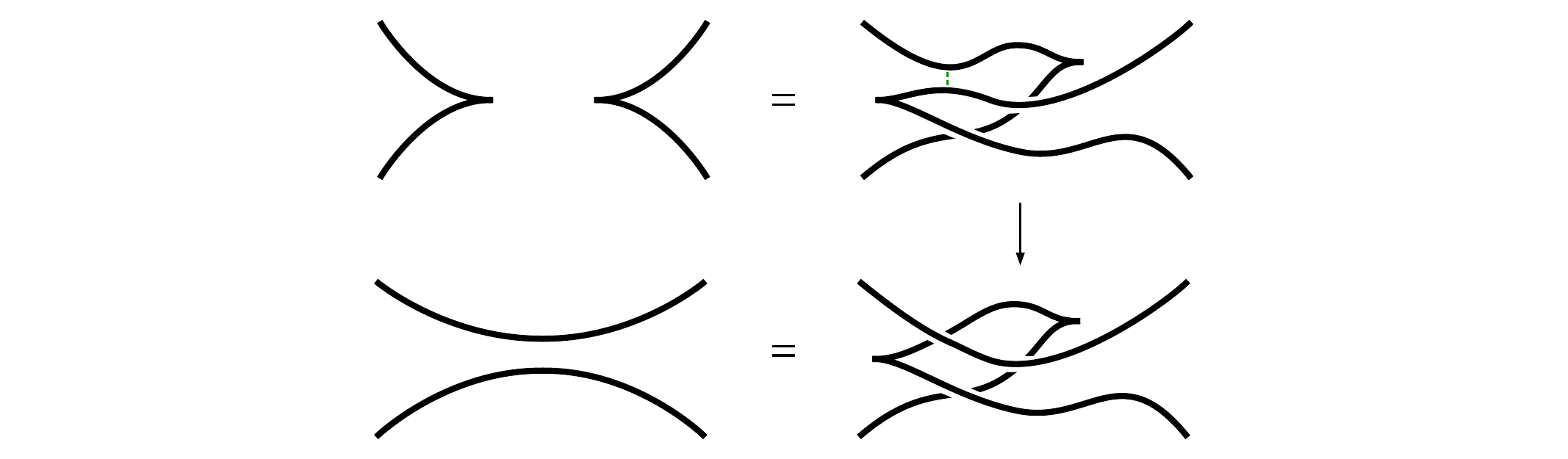}
       \put(66,14){\footnotesize Type-$R$ band}
       
	\end{overpic}
	\caption{Locally executing a Type-$D$ band via a Type-$R$ band.}
	\label{fig:DfromR}
\end{figure}

\begin{figure}[ht]
	\centering
	\begin{overpic}[scale=.26]{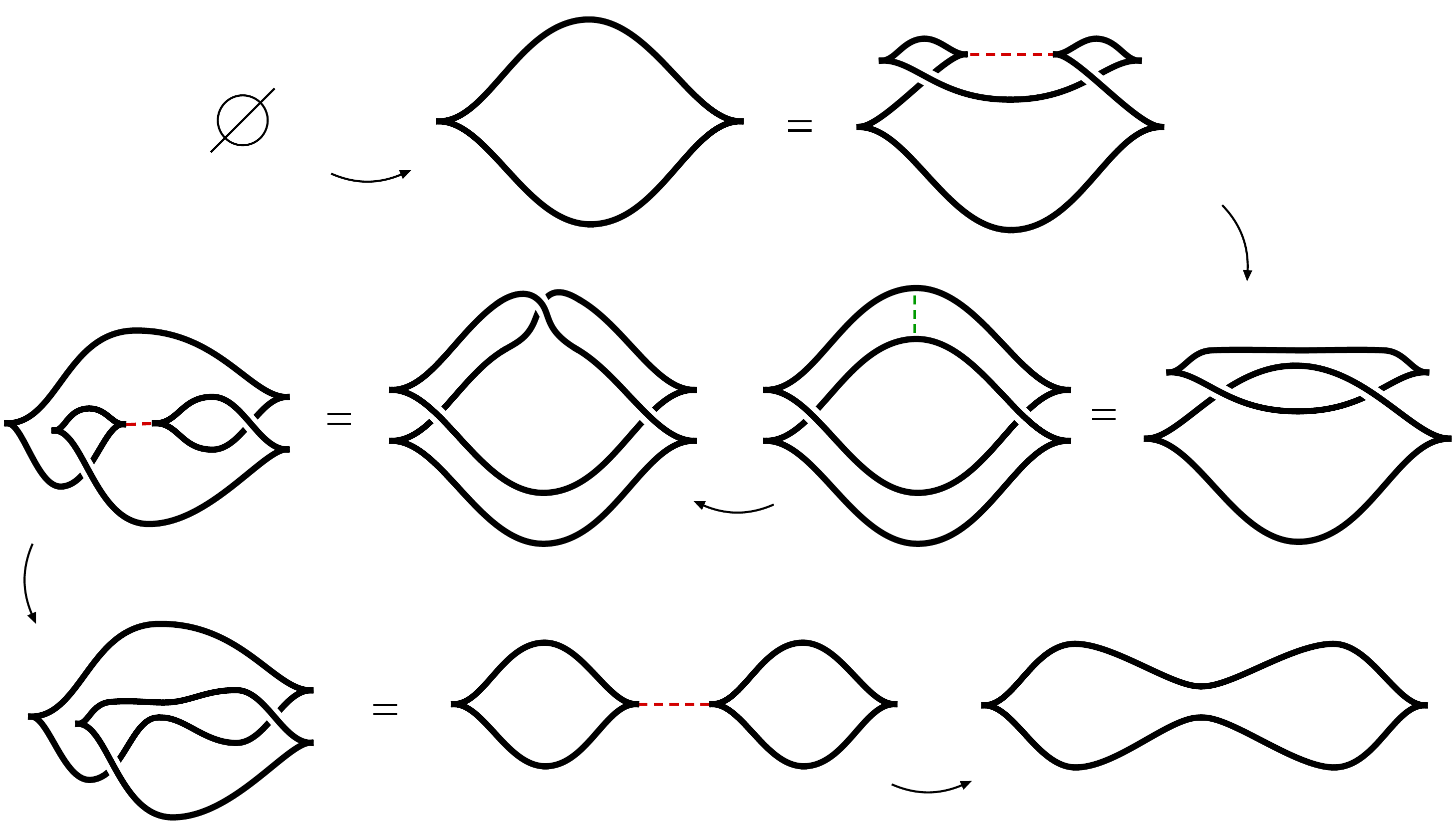}
    
	\end{overpic}
	\caption{A Type-$R$ band that cannot be recovered from a Type-$D$ band.}
	\label{fig:nonorientable}
\end{figure}

Informally, the following lemma says that band sums and handleslides performed along dividing curve arcs are naturally of Type-$R$. 

\begin{lemma}\label{lemma:div_curve_h_slide}
Let $\Sigma$ be a convex surface with distinct boundary components $L_{\pm}$ and a dividing curve arc $\Gamma$ connecting $L_-$ to $L_+$. Then the band sum of $L_-, L_+$ along $\Gamma$, framed by $\Sigma$, is given by a Type-$R$ band sum. In particular, if $L_-$ is $0$-framed by $\Sigma$, then the resulting band sum is Legendrian isotopic to a contact-$(+1)$ handleslide of $L_+$ over $L_-$. 
\end{lemma}

\begin{proof}
Let $N\subset \Sigma$ be a small tubular neighborhood of $\Gamma$ in $\Sigma$. Endow $N\cong [-1,1]_p \times [-2,2]_q$ with local coordinates $(p,q)$ so that $\Gamma = [-1,1]_p\times\{0\}$ and $L_{\pm} \cap N = \{\pm 1\}\times [-2,2]_q$. Define arcs $\ell_{\pm 1}\subset N$ as $\ell_{\pm 1} = [-1,1]_p\times \{\pm 1\}$; see the left side of \cref{fig:model}.

\begin{figure}[ht]
	\centering
    \vskip-0.8cm
	\begin{overpic}[scale=.38]{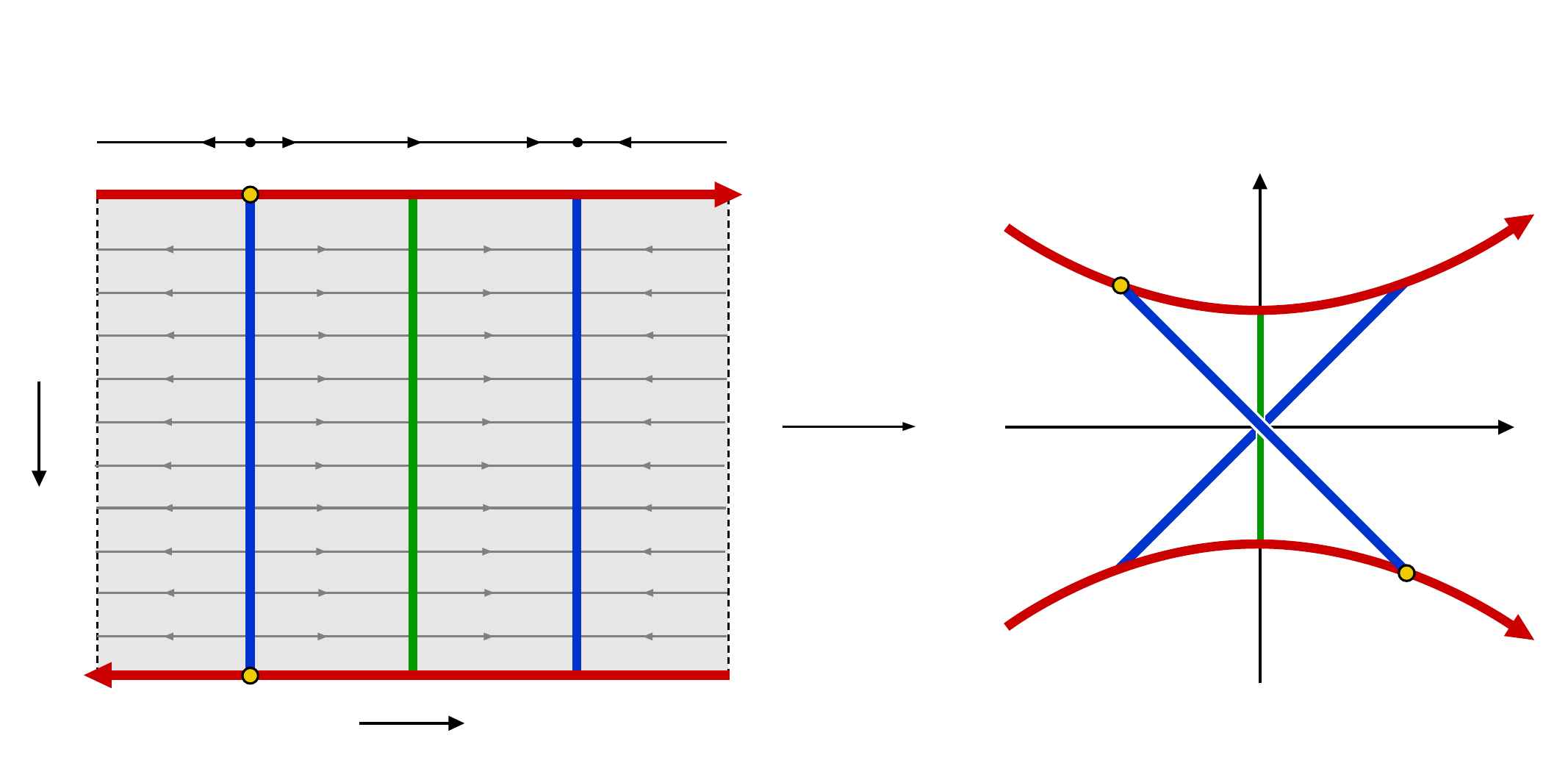}
        \put(2.75,36){\small \textcolor{darkred}{$L_-$}}
        \put(63,36){\small \textcolor{darkred}{$L_-$}}
        \put(62.5,6){\small \textcolor{darkred}{$L_{+}$}}
        \put(47.5,4.5){\small \textcolor{darkred}{$L_{+}$}}
        \put(2,16){\small $p$}
        \put(30,2){\small $q$}
        \put(53,23){\small $\varphi$}
        \put(97,21){\small $x$}
        \put(79.75,38.5){\small $z$}
        \put(15,3){\small \textcolor{darkblue}{$\ell_{-1}$}}
        \put(36,3){\small \textcolor{darkblue}{$\ell_{1}$}}
        \put(69,14){\small \textcolor{darkblue}{$\ell_{1}$}}
        \put(68.5,28){\small \textcolor{darkblue}{$\ell_{-1}$}}
	\end{overpic}
	\caption{}
	\label{fig:model}
\end{figure}

Note that $\ell_{-1} \cup \ell_{1}$ is non-isolating with respect to the dividing set of $\Sigma$; indeed, there must be some other arc of the dividing set present by Thurston-Bennequin considerations. By the Legendrian realization principle, we may assume after an isotopy of $\Sigma$ that $\ell_{-1} \cup \ell_{1}$ is Legendrian. In particular, Giroux flexibility allows us to witness the characteristic foliation in $N$ indicated in gray (and projected in black above $N$) in \cref{fig:model}, which is printed on $N =\{t=0\}$ by the $\R$-invariant contact structure 
\[
\left([-1,1]_p \times [-2,2]_q \times \R_t, \, \alpha = -q\, dt + (q^2 - 1)\, dp\right).
\]
Now we identify a strict contactomorphism of this invariant neighborhood $N\times \R$ with a neighborhood of the origin in a Darboux chart $(\R^3, dz-y\, dx)$ in order to view the Legendrian complex $L_-\cup L_{+}\cup \ell_{\pm 1}$ in the front projection. One may verify that the map $\varphi: N\times \R \to \R^3_{(x,y,z)}$ given by 
\[
\varphi(p,q,t) = (t-2pq, \, q, \, -p(q^2+1))
\]
satisfies $\varphi^*(dz-y\, dx) = \alpha$, hence is the desired contactomorphism. 

Next, we identify the image of the Legendrian complex $L_-\cup L_{+}\cup \ell_{\pm 1}$ under $\varphi$. Letting $L_-$ denote the upper arc on the left of \cref{fig:model}, we may parametrize it by $\tau \mapsto (p=-1, q=\tau,t=0)$. Composing with $\varphi$ gives the parametrization 
\[
L_-: \quad \tau \mapsto (x= 2\tau,y=\tau,z=\tau^2 + 1).
\]
The resulting front projection to the $(x,z)$-plane is the given by the arc on the right of \cref{fig:model}. One may additionally verify the following front projection parametrizations: 
\begin{align*}
    L_{+}:&\quad \tau \mapsto (2\tau,\tau,-(\tau^2 + 1))\\
    \ell_{-1}:&\quad \tau \mapsto (2\tau,-1,-2\tau)\\
    \ell_{1}:&\quad \tau \mapsto (2\tau,1,2\tau)\\
    \Gamma:&\quad \tau \mapsto (0,0,-\tau)
\end{align*}
which produce the right side of \cref{fig:model}. This corresponds to a Type-$R$ band and the standard model for a contact-$(+1)$ handleslide. 
\end{proof}

\section{Legendrian derivative links}\label{sec:reg_char}

In this section we prove the derivative characterizations \cref{thm:legKauffman} and \cref{thm:decKauffman}.

\subsection{Derivative characterization of regular sliceness}\label{subsec:reg_char}

We begin with \cref{thm:legKauffman}. The proof is a contact topological adaptation of the proof of \cite[Theorem 1.3]{miller2023handle}, appropriately enriched with convex surface theory.

\begin{definition}\label{def:standard_planar}
Let $\Sigma\subset (Y, \xi)$ be a connected planar surface. We call $\Sigma$ a \emph{standard convex planar surface} if $\partial \Sigma$ is Legendrian, $\Sigma$ is convex, $\Gamma$ is the union of $|\partial \Sigma|$ arcs, with each boundary component meeting exactly two arcs, and $\Gamma$ cuts $\Sigma$ into two disks, as shown in \cref{fig:standard}.
\end{definition}

\begin{figure}[ht]
	\centering
	\begin{overpic}[scale=.23]{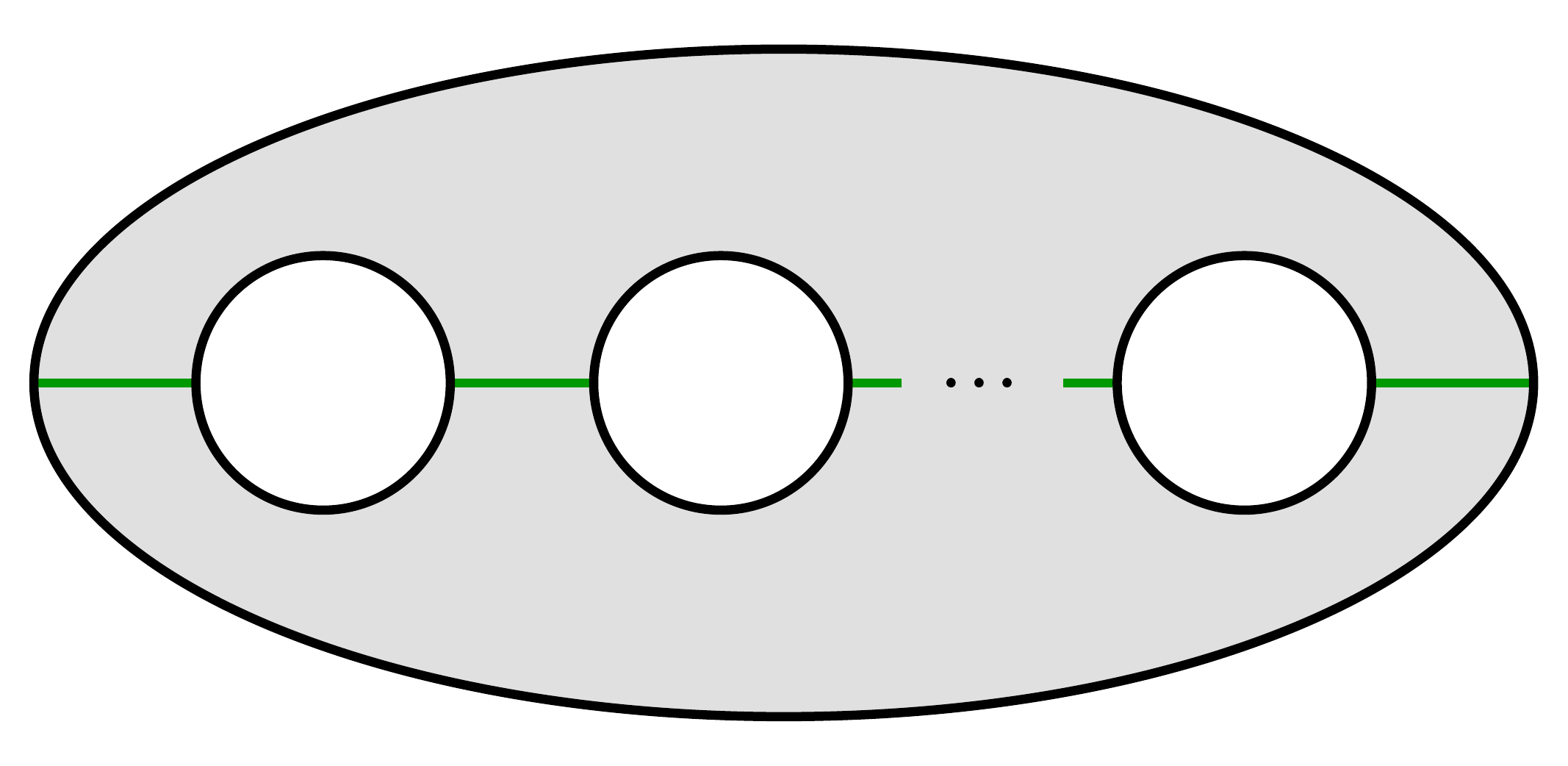}
        \put(8,39.5){$\Sigma$}
        \put(-2,23.5){\small \textcolor{darkgreen}{$\Gamma$}}
	\end{overpic}
	\caption{A standard convex planar surface with dividing curves $\Gamma$.}
	\label{fig:standard}
\end{figure}

\begin{lemma}\label{lemma:planar_surface_bounds_R_link_new}
Let $\Sigma\subset (S^3, \xi_{\mathrm{st}})$ be a connected convex planar surface with Legendrian boundary inducing the $0$-framing of each boundary component. Assume $\mathrm{tb}=-1$ for each component of $\partial \Sigma$. If contact-$(+1)$ surgery along $\partial \Sigma$ is tight, then $\Sigma$ is a standard convex planar surface.    
\end{lemma}

\begin{proof}
Let $L = \partial \Sigma$ and $L' = \Sigma \cap \partial N(L)$, where $N(L)$ is a small tubular neighborhood of $L$, so that $L'$ is a $0$-framed push-off of $L$. Since $L$ is Legendrian and $\Gamma_{\Sigma}\pitchfork L$, if $N(L)$ is sufficiently small we may assume by the Legendrian realization principle that $L'$ is Legendrian and Legendrian isotopic to $L$. 

Perform contact-$(+1)$ surgery, which induces smooth $0$-surgery, along $L$ via $N(L)$. By assumption, the resulting contact manifold is tight. Each component $L_i'$ of the $0$-framed push-off $L'$ bounds a meridian disk in the corresponding surgery torus, and has $\mathrm{tb}(L_i') = -1$. By tightness and Giroux's criterion each meridian disk may be taken to be convex with a single dividing curve. Using these disks, which agree with the $0$-framing induced on $L'$ by $\Sigma$, we cap off the surface $\Sigma - N(L)$ in the surgered manifold to obtain a convex sphere $S^2$.

Again appealing to tightness and Giroux's criterion, the resulting convex sphere $S^2$ has a single equatorial dividing curve. Moreover, $\Sigma$ is obtained by deleting the meridian disks, each of which has a single dividing curve arc. It follows that $\Sigma$ is a standard convex planar surface. 
\end{proof}

\begin{lemma}\label{lemma:reg_slice_R_link_component}
A Legendrian knot $\Lambda \subset (S^3, \xi_{\mathrm{st}})$ is regularly slice if and only if it is a component of some Legendrian $R$-link.     
\end{lemma}

\begin{proof}
Suppose that $\Lambda$ is regularly slice. By assumption there is a Weinstein structure on $B^4$, homotopic to the radial structure, for which $\Lambda$ is the belt sphere of a $2$-handle. Let $L$ denote the collection of remaining $2$-handle belt spheres. Contact-$(+1)$ surgery on $\Lambda \cup L$ is witnessed by the removal of all $2$-handles. At the $4$-dimensional level, this leaves $\natural^{|\Lambda \cup L|} (S^1 \times B^3)$ with its standard Weinstein structure and tight boundary $\#^{|\Lambda \cup L|} (S^1\times S^2)$. Hence, $\Lambda \cup L$ is a Legendrian $R$-link. 

The converse direction follows from the construction described in the the proof of \cref{thm:stable_leg_gen_prop_r}. Namely, suppose that $\Lambda$ is a component of some Legendrian $R$-link $\Lambda \cup L$. Letting $Y = S^3_{\Lambda \cup L}((+1)) \cong \#^{|\Lambda \cup L|} (S^1\times S^2)$, we first build a Weinstein cobordism from $Y$ to $S^3$ by attaching Weinstein $2$-handles along a contact push-off $\Lambda' \cup L'$ of $\Lambda \cup L$. We promote this to a Weinstein filling of $S^3$, necessarily deformation equivalent to the radial Weinstein structure on $B^4$, by filling $Y$ with $\natural^{|\Lambda \cup L|} (S^1 \times B^3)$. After a possible self-diffeomorphism of $B^4$, the resulting Weinstein structure is homotopic to the radial structure, by uniqueness of such fillings \cite{cieliebak2012stein}. 

We claim that the co-core of the $2$-handle attached along the push-off of $\Lambda$ yields a regular slice disk for $\Lambda$. Indeed, the belt sphere is isotopic to a further contact push-off $\Lambda''$ of $\Lambda'$ in $S^3$. Canceling all contact-$(\pm 1)$ surgeries leaves $\Lambda''\subset S^3$, which is Legendrian isotopic to $\Lambda$. 
\end{proof}

For a smooth knot $K\subset S^3$, a \emph{partial derivative} supported by a Seifert surface $F$ is a link $J\subset F$ such that components of $J$ are pairwise non-homotopic in $F$ and on which the Seifert form vanishes. (A derivative is further required to cut $F$ into a connected planar surface.) The final lemma is the contact enrichment of \cite[Proposition 3.2]{miller2023handle} and upgrades a Legendrian $R$-link partial derivative to a Legendrian $R$-link derivative. Compared to the smooth statement, there are additional contact complications and the proof is somewhat more involved. 

\begin{lemma}\label{lemma:pd_to_der}
Let $\Lambda\subset (S^3, \xi_{\mathrm{st}})$ be a Legendrian knot. Assume $\Lambda$ admits a Seifert surface $F$ supporting a Legendrian partial derivative $L$ such that $\Lambda \cup L$ is a Legendrian $R$-link. Then $\Lambda$ admits a Seifert surface supporting a Legendrian $R$-link derivative.   
\end{lemma}

\begin{proof}
The Legendrian partial derivative may not cut $F$ into a planar surface. In this case, \textsc{Step 1A} and \textsc{Step 1B} below describe how to produce a new Seifert surface $F'$ supporting a Legendrian partial derivative $L'$ such that $\Lambda \cup L'$ is a Legendrian $R$-link and $g(F' - L') < g(F - L)$ (for a disconnected surface $S$, $g(S)$ is defined here as a tuple of the genera of its connected components, listed in non-increasing order, and compared with the dictionary ordering). Once planarity of the cut open Seifert surface is ensured, \textsc{Step 2} describes how to extract a Legendrian $R$-link derivative, completing this proof. 

First, assume $F - L$ is not planar. 

\vspace{2mm}
\textsc{Step 1A:} \textit{Enlarge $L$ to $L \cup J$ so that $g(F - (L\cup J)) < g(F - L)$.}
\vspace{2mm}

Performing contact-$(+1)$ surgery on $\Lambda \cup L$ gives tight $\#^{|\Lambda \cup L|}(S^1 \times S^2)$. As in the proof of \cref{lemma:planar_surface_bounds_R_link_new}, we cap off the surface $F - N(\Lambda \cup L)$ with meridian disks in the surgery tori (one disk in the surgery torus corresponding to $\Lambda$, and two disks in each torus corresponding to each component of $L_i$) to obtain a closed, possibly disconnected, convex surface $\hat{F}\subset \#^{|\Lambda \cup L|}(S^1 \times S^2)$. We abuse notation and use $F$ to refer both to the natural subsurface $F- N(\Lambda \cup L)\subset \hat{F}\subset \#^{|\Lambda \cup L|}(S^1 \times S^2)$, and also to the original surface in $S^3$. 

The rest of this step is essentially the same as the argument in the proof of \cite[Proposition 3.2]{miller2023handle}. In $\#^{|\Lambda \cup L|}(S^1 \times S^2)$, every surface with positive genus is compressible. Since $F - L$ is not planar, $\hat{F}$ contains a component with positive genus, hence there is a compressing disk $D\subset \#^{|\Lambda \cup L|}(S^1 \times S^2)$ for $\hat{F}$ with boundary $J = \partial D\subset \hat{F}$. We can assume up to isotopy that $J$ is disjoint from $L$, since $L$ is homotopically trivial in $\hat{F}$. 

We claim that $L\cup J \subset F \subset S^3$ is a smooth partial derivative for $\Lambda$, and that $\Lambda \cup L \cup J$ is a topological $R$-link. Since $J$ bounds a compressing disk for $\hat{F}$, it is unknotted, unlinked with $L$, and $0$-framed by $F$. Thus, $L\cup J \subset F \subset S^3$ is a smooth partial derivative for $\Lambda$. Moreover, since $J$ is unknotted in $\#^{|\Lambda \cup L|}(S^1 \times S^2)$, 
\[
S^3_{\Lambda \cup L \cup J}(0) = (\#^{|\Lambda \cup L|}(S^1 \times S^2))_J(0) = \#^{|\Lambda \cup L\cup J|}(S^1 \times S^2).
\]
It follows that $\Lambda \cup L \cup J$ is a topological $R$-link. By construction, $g(F - (J \cup L)) < g(F-L)$.  We next need to arrange for the Legendrian $R$-link condition.

\vspace{2mm}
\textsc{Step 1B:} \textit{Stabilize $F$ and refine $J$.}
\vspace{2mm}

We may assume after a generic perturbation of $J$ that $J\subset F \subset \#^{|\Lambda \cup L|}(S^1 \times S^2)$ intersects the dividing curves $\Gamma_F$ of $F$ transversely, hence is Legendrian by the Legendrian realization principle. By tightness, unknottedness of $J$ in $\#^{|\Lambda \cup L|}(S^1 \times S^2)$, and the fact that $J$ is $0$-framed by $F$, $J$ necessarily has nonempty intersection with the dividing curves.  If $|J \cap \Lambda_{\hat F}| = 2$, then $J$ is a max-tb unknot.  Otherwise, suppose $|J \cap \Gamma_{\hat{F}}| > 2$. We first describe how to remedy the situation on $\hat{F}$ in $\#^{|\Lambda \cup L|}(S^1 \times S^2)$.

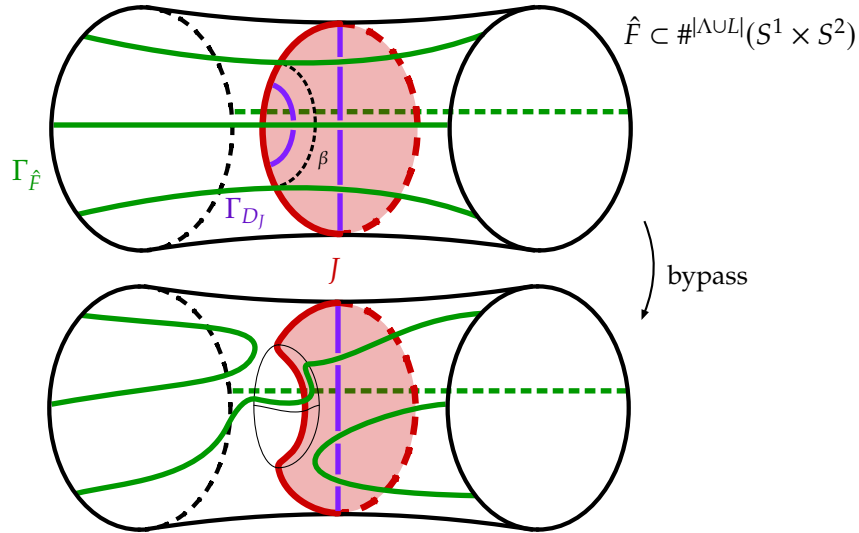
\begin{figure}[ht]
	\centering
	\begin{overpic}[scale=.4]{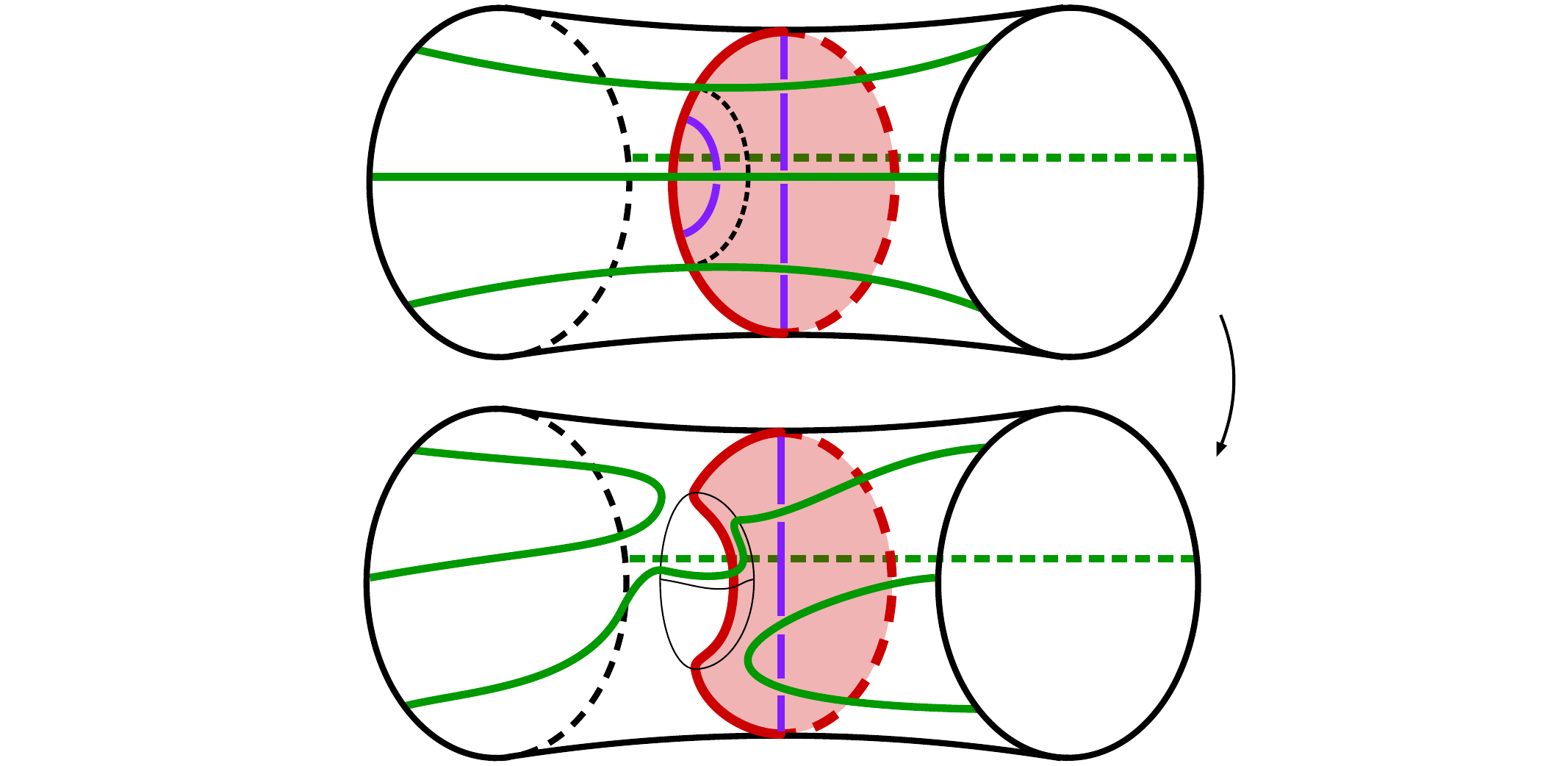}
        \put(80,23){\small bypass}
        \put(76,45){$\hat{F}\subset \#^{|\Lambda \cup L|}(S^1 \times S^2)$}
        \put(49,23.5){\textcolor{darkred}{$J$}}
        \put(39.5,29){\textcolor{darkpurple}{$\Gamma_{D_{J}}$}}
        \put(48,34){\tiny $\beta$}
        \put(20,32.5){\textcolor{darkgreen}{$\Gamma_{\hat{F}}$}}
	\end{overpic}
	\caption{A schematic picture for the isotopy of $\hat{F}$ when $\mathrm{tb}(J)= -2$. Dividing curves for the compressing disk $D_{J}$ are in purple, while those of $\hat{F}$ are in green. The bypass half-disk in $D_{J}$ is outlined by the dashed black arc $\beta$ in the top level.}
	\label{fig:bypassdisk}
\end{figure}

We may assume after a perturbation that the compressing disk $D_J$ is convex. Since $\#^{|\Lambda \cup L|} (S^1 \times S^2)$ is tight, Giroux's criterion implies that the dividing curves of $D_{J}$ consist of $n>1$ arcs, where $|J \cap \Gamma_{\hat{F}}| = 2n$ and thus $\mathrm{tb}(J) = -n$. By \cite[Proposition 3.18]{honda2000classification}, each outermost boundary parallel dividing curve on $D_{J}$ gives rise to a bypass half-disk for $\hat{F}$. By isotoping $\hat{F}$ in $\#^{|\Lambda \cup L|}(S^1 \times S^2)$ past each half-disk and carrying along $J$ through the isotopy, it is possible to destabilize $J$ until $\mathrm{tb}(J)=-1$. Such an isotopy of $\hat{F}$ modifies the dividing curves of $\hat{F}$ near $J$ so that $J$ intersects them twice. See \cref{fig:bypassdisk}. 

The issue is that this isotopy of $\hat{F}$ inside $\#^{|\Lambda \cup L|}(S^1 \times S^2)$ may not be induced by an isotopy of $F$ inside $S^3$, as the compressing disk only exists in the former; said differently, in $\#^{|\Lambda \cup L|}(S^1 \times S^2)$ there are dual surgery curves that possibly intersect the compressing disks. Instead of isotoping $\hat{F}$ past the bypass half-disk, we will stabilize $F$ (inside $S^3$) using the boundary of the contact $1$-handle associated to a neighborhood of the bypass half-disk (see the refinement process of \cite[\S 5]{licata2024heegaard}) and exchange the destabilization of $J$ for a fracturing process that adds additional components to the partial derivative.  

To do this, first let $N\subset \#^{|\Lambda \cup L|}(S^1 \times S^2)$ denote the union of solid surgery tori, so that 
\[
\#^{|\Lambda \cup L|}(S^1 \times S^2) - N = S^3 - N(\Lambda \cup L).
\]
Note that $N$ is a standard tubular neighborhood of the dual surgery Legendrian link $\Lambda^* \cup L^*$ in $\#^{|\Lambda \cup L|}(S^1 \times S^2)$. We may assume that $\Lambda^*\cup L^*$ is transverse to the interior of $D_{J}$. Furthermore, by shrinking $N$ if necessary we may assume that $N \cap D_{J}$ is disjoint from the boundary of each bypass half-disk (e.g.\ the dashed black arc labeled $\beta$ in \cref{fig:bypassdisk} and \cref{fig:bypassdisk2}.\footnote{In fact, we can assume that $\Lambda^*\cup L^*$ intersects the dividing set, as described in \cite[Lemma 4.2]{breen2024regularlysliceimpliesoncestably}.}  

\begin{figure}[ht]
	\centering
	\begin{overpic}[scale=.4]{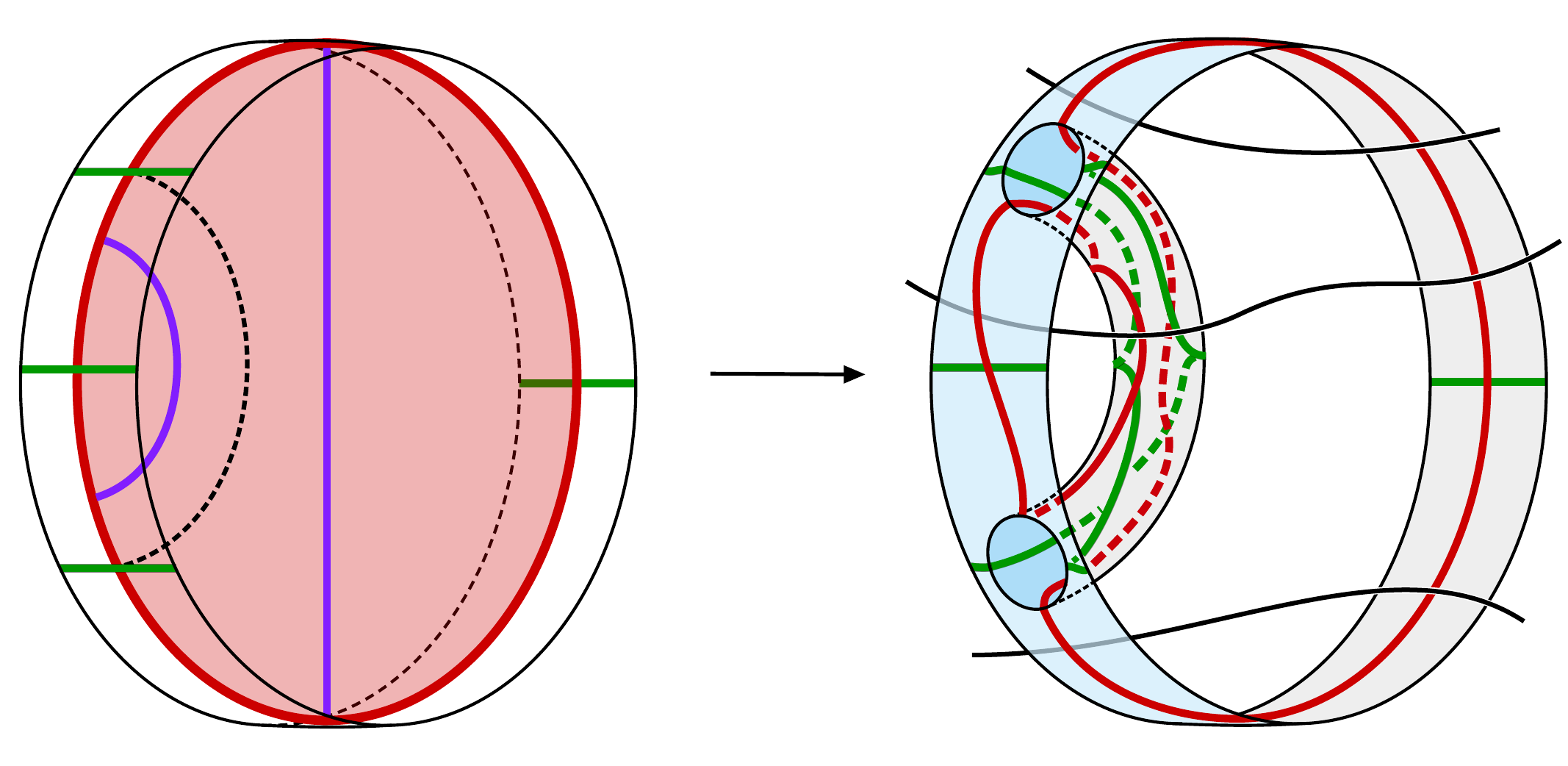}
        \put(78.5,13){\small $\Lambda^* \cup L^*$}
        \put(2,40){\small $F$}
        \put(16,30){\small $\beta$}
        \put(40,35){\small \textcolor{darkred}{$J$}}
        \put(60,40){\small $F'$}
        \put(58,35){\small \textcolor{darkred}{$J'$}}
        \put(84,35){\small \textcolor{darkred}{$J''$}}
	\end{overpic}
	\caption{Stabilizing $F\subset S^3$ with a tube along the arc $\beta$ of the bypass half-disk.}
	\label{fig:bypassdisk2}
\end{figure}
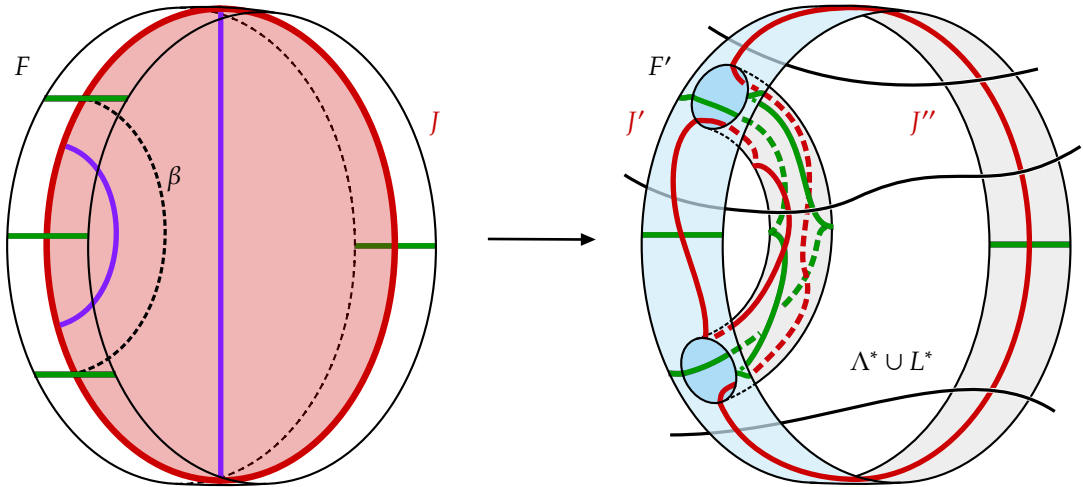

Let $\beta$ be a Legendrian arc properly embedded in $D_{J}$ (partially) bounding a bypass-half disk, and let $N(\beta)$ be a solid standard tubular neighborhood with convex boundary, i.e.\ a contact $1$-handle with core $\beta$. As $\beta$ bounds a bypass half-disk, and the boundary of a bypass half-disk is a max-tb unknot, the contact structure (and hence the dividing curves of $\partial N(\beta)$) make half of a left-handed twist relative to $D_{J}$ along $\beta$. Let $F'$ denote the surface obtained by attaching $\partial N(\beta)$ to $F$. By taking $N(\beta)$ sufficiently small, we may assume $F' \subset S^3$. Now, let $J' \cup J'' = D_{J} \cap F'$, and replace $J$ by with $J' \cup J''$, letting $L' = L \cup J' \cup J''$; see \cref{fig:bypassdisk2}. We make a number of observations. 

\begin{enumerate}
    \item The pair $J'\cup J''$ is an unlink in $\#^{|\Lambda \cup L|}(S^1 \times S^2)$, as $J'$ and $J''$ bound separate subdisks of $D_{J}$. Moreover, each component is $0$-framed by $F'$.

    \item Both $J'$ and $J''$ are $0$-framed in $S^3$ by $F'$, and are algebraically unlinked in $S^3$. This is immediate in $\#^{|\Lambda \cup L|}(S^1 \times S^2)$ by the previous observation. The conclusion in $S^3$ follows from the fact that $J'$ and $J''$ can be deformed to unlinked unknots in $S^3$ by isotoping them through strands of $\Lambda' \cup L'$ in  $\#^{|\Lambda \cup L|}(S^1 \times S^2)$, the dual surgery link that intersects $D_{J}$ transversely. By \cite[Lemma 2.4]{miller2023handle}, this implies that $J'$ and $J''$ are unknotted and unlinked in $S^3$ up to $0$-framed handleslides across $\Lambda \cup L$. Such handleslides preserve the framing and algebraic linking number.    

    \item We have $|J' \cap \Gamma_{\hat{F}}| = 2$ and $|J'' \cap \Gamma_{\hat{F}}| = 2n-2$, hence $\tb(J') = -1$ and $\tb(J'') = -n+1$. This is a consequence of the above remark on the contact framing of the arc $\beta$. Alternatively, one can make the same conclusion by appealing to the Seifert subdisks of $D_{J}$ in $\#^{|\Lambda \cup L|}(S^1 \times S^2)$.

    \item The surface $F' - (L \cup J' \cup J'')$ has the same genus as $F - (L \cup J)$. To see this, it suffices to verify that deleting $J' \cup J''$ from the local subsurface of $F'$ on the right side of \cref{fig:bypassdisk2} results in a planar surface. This is evident by inspection.
\end{enumerate}

We can cycle through \textsc{Step 1B} again, this time applied to $J''$, to continue fracturing the partial derivative and increasing $\tb$, renaming $F'$ and $L'$ at each step. We continue until each component added to $L$ associated to the initial compressing disk $D_J$ has $\tb = -1$. Thus, after cycling through \textsc{Step 1B} sufficiently many times, we obtain a Seifert surface $F'$ supporting a Legendrian partial derivative $L'$ such that $\Lambda \cup L'$ is a Legendrian $R$-link and $g(F' - L') < g(F - L)$. 

If $F'- L'$ is still not planar, we return to \textsc{Step 1A} to find another compressing disk and repeat the process. Thus, after cycling through the entirety of \textsc{Step 1} sufficiently many times, we obtain a Seifert surface $F'$ supporting a Legendrian partial derivative $L'$ such that $\Lambda \cup L'$ is a Legendrian $R$-link and $F' - L'$ is planar. 

\vspace{2mm}
\textsc{Step 2:} \textit{Extract a derivative on $F'$ from $L'$.}
\vspace{2mm}

To begin, note that, in general, if $\tilde{L} \cup \tilde{U}$ is a Legendrian $R$-link and $\tilde{U}$ is a split max-tb unlink, then $\tilde{L}$ is a Legendrian $R$-link. Additionally, the property of being a Legendrian $R$-link is preserved under contact-$(+1)$ handleslides. Thus, to extract a Legendrian $R$-link derivative supported by $F'$ from $L'$, we perform contact-$(+1)$ handleslides in $F'$ among the components of $L'$ to a produce a split link $\tilde{L} \sqcup \tilde{U}$, where $\tilde{U}$ is a max-tb unlink and $F' - L'$ is connected; the link $\tilde{L}$ is the resulting Legendrian $R$-link derivative. 

To see that such a sequence of handleslides is possible, note that each connected (planar) component of $F' - L'$ is a standard convex planar surface by \cref{lemma:planar_surface_bounds_R_link_new}. Using the dividing arcs as a guide, we apply \cref{lemma:div_curve_h_slide} to contact-$(+1)$ handleslide any disconnecting components to unlinked max-tb unknots; see \cref{fig:unknottingslide}. This completes the proof. 
\end{proof}

\begin{figure}[ht]
	\centering
	\begin{overpic}[scale=.25]{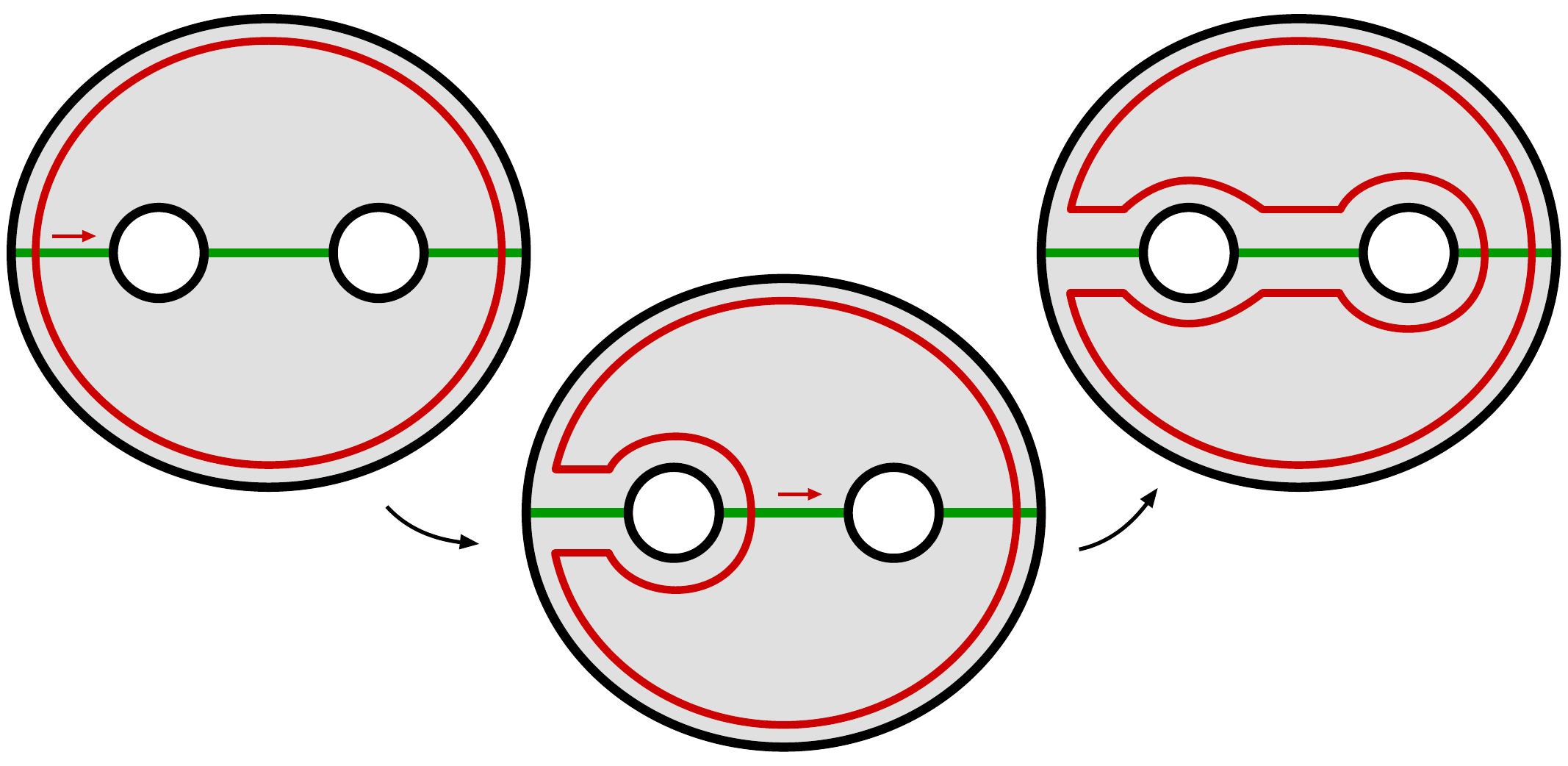}
        
	\end{overpic}
	\caption{Sliding across derivative components to produce max-tb unknots.}
	\label{fig:unknottingslide}
\end{figure}

\begin{proof}[Proof of \cref{thm:legKauffman}.]
First, assume that $\Lambda$ is regularly slice. By \cref{lemma:reg_slice_R_link_component}, $\Lambda$ is a component of a Legendrian $R$-link $\Lambda \cup L$. By the proof of \cite[Theorem 1.3]{miller2023handle}, we may choose a Seifert surface $F_0$ for $\Lambda$ which avoids $L$. For each (necessarily $\mathrm{tb}=-1$) component $L_i$ of $L$, let $T^2_i$ denote the oriented boundary of a sufficiently small standard neighborhood of $L_i$, and $L_i'\subset T^2_i$ a Legendrian push-off of $L_i$. 

Next we use round $1$-handles to connect the tori $T_i^2$ to $F_0$. Specifically, for each $i$ let $\gamma_i$ be a Legendrian arc satisfying the following properties: 
\begin{enumerate}
    \item One endpoint of $\gamma_i$ is on the dividing set of $F_0$, while the other is on the dividing set of $T_i^2$ away from $L_i'$.  
    \item The convex boundary of a small standard neighborhood of $\gamma_i$ forms an orientable round $1$-handle attachment between $F_0$ and $T_i^2$. 
\end{enumerate}
Note that we may also describe the tubing between $F_0$ and $T^2_i$ as the boundary of a contact $1$-handle attachment with core $\gamma_i$ to a standard invariant neighborhood of $F_0$ and the neighborhood of $L_i$. 

Let $F$ be the resulting convex Seifert surface for $\Lambda$ obtained by tubing each $T_i^2$ to $F_0$ along the $\gamma_i$ arcs. Let $L' = \bigcup_i L_i'$. Since $L'$ is Legendrian isotopic to $L$, $\Lambda \cup L'$ is a Legendrian $R$-link by assumption. By construction, $L'$ is moreover a partial derivative for $\Lambda$ supported on $F$. Thus, by \cref{lemma:pd_to_der}, $\Lambda$ admits a Legendrian $R$-link derivative in some Seifert surface $F'$.   

Next we prove the converse implication. Suppose that $\Lambda$ has $\mathrm{tb}(\Lambda) = -1$ and admits a tight $R$-link derivative $L$ in a Seifert surface $F$. We claim that $\Lambda \cup L$ is a Legendrian $R$-link, which by \cref{lemma:reg_slice_R_link_component} implies that $\Lambda$ is regularly slice. In fact, we will argue that $\Lambda$ is a max-tb unknot in tight $Y := S^3_{L}((+1)) \cong \#^{|L|}(S^1\times S^2)$, which then implies the chain of (tight) contactomorphisms
\[
S^3_{\Lambda \cup L}((+1)) \cong Y_{\Lambda}((+1)) \cong (S^1 \times S^2) \,\#\, Y \cong \#^{|\Lambda \cup L|}(S^1 \times S^2)
\]
and consequently the fact that $\Lambda \cup L$ is a Legendrian $R$-link. 

To see that $\Lambda$ is a max-tb unknot in $Y$, again as in the proof of \cref{lemma:planar_surface_bounds_R_link_new} we may construct a convex Seifert disk for $\Lambda$ in $Y = S_{L}((+1))$ by gluing on meridional disks in the surgery tori corresponding to $L$ to the planar surface $F - L$. The framing of the resulting disk on $\Lambda$ agrees with the $F$-framing. As $\mathrm{tb}(\Lambda) = -1$ by assumption, we conclude the proof.
\end{proof}

\subsection{Unlink derivatives for decomposably slice knots} Our focus so far has been establishing regular Lagrangian sliceness rather than decomposability, the primary reason being \cref{lemma:div_curve_h_slide}. The derivative characterization in \cref{thm:legKauffman} is rooted in convex surface theory, and the referenced lemma establishes that the convex surface model naturally gives constructions and manipulations of Legendrians which arise from band sums incompatible with the usual formulation of decomposable cobordisms in the front. Nevertheless, concerning decomposability, here we prove \cref{thm:decKauffman}.

First, we recall the idea behind the relevant direction of \cref{lemma:ribbon_derivative}. Given a ribbon disk $D$ with $g$ ribbon singularities $\gamma_1, \dots, \gamma_g$, we obtain an unlink derivative on an embedded Seifert surface as follows. For each ribbon singularity $\gamma_i$, let $N(\gamma_i)$ be a tubular neighborhood and let $U_i\subset D$ be the unknot which is the closed component of $N(\gamma_i) \cap D$. Along each ribbon singularity, perform an oriented "cellar-door resolution" as depicted in \cref{fig:cellar}. The resulting embedded Seifert surface has genus $g$ and the unlink $U := U_1 \cup \cdots \cup U_g$ is a derivative. As a shorthand, we will refer to the unlink $U$ embedded in $D$ (before resolution) as a \emph{ribbon derivative}.

\begin{figure}[ht]
	\centering
	\begin{overpic}[scale=.35]{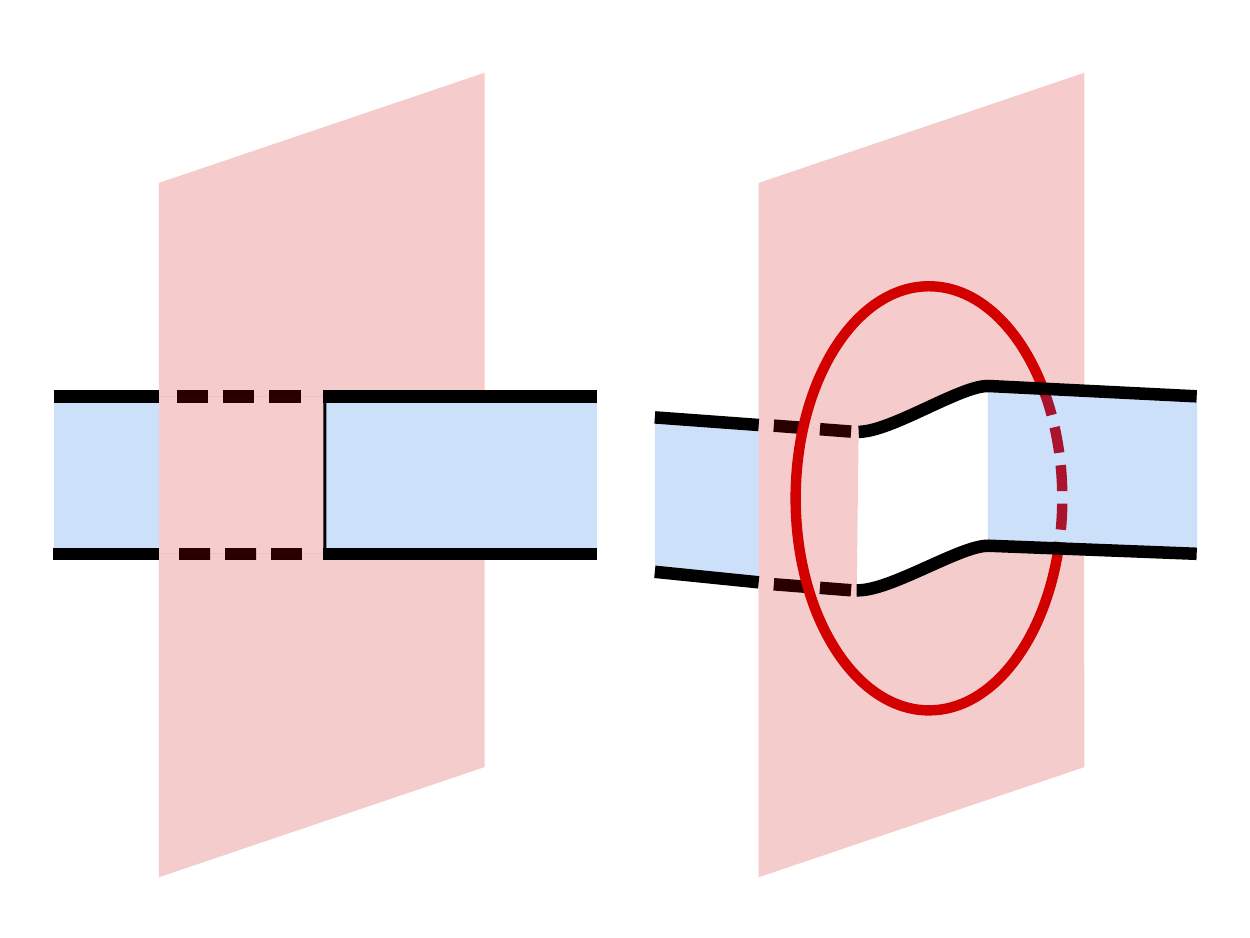}
       
	\end{overpic}
	\caption{A cellar-door resolution with a component of a ribbon derivative.}
	\label{fig:cellar}
\end{figure}

\begin{proof}[Proof of \cref{thm:decKauffman}.] 
Since $\Lambda$ is decomposably slice, it is obtained from a sequence of max-tb unknot births, ambient Legendrian surgeries along embedded Legendrian arcs, and Legendrian isotopies. We may arrange for all unknot births to occur before any requisite isotopies and surgeries, so we begin with an max-tb unlink $U_1 \cup \cdots \cup U_n$ bounding $n$ disconnected Seifert disks. We will argue that the sequence of decomposable moves producing $\Lambda$ can be arranged to give a ribbon disk $D$ with a max-tb unlink ribbon derivative.

We proceed inductively. Assuming that we have constructed a (possibly disconnected) ribbon disk with a max-tb unlink ribbon derivative, we will show that the next step of the sequence (either Legendrian isotopy or ambient Legendrian surgery) yields another ribbon disk with a max-tb unlink ribbon derivative. This is immediate with any Legendrian isotopy, as Legendrian isotopy coincides with ambient contact isotopy. Thus, it remains to consider an ambient Legendrian surgery. 

After a possible isotopy of the Legendrian surgery arc, we may assume that its interior is transverse to the ribbon disk constructed thus far, and that a neighborhood of the arc is given by the left side of \cref{fig:dec-derivative}. For each transverse intersection of the arc with the surface, identify a small meridian max-tb unknot and perturb the transverse sheet to support the unknot. After attaching the associated band, we obtain a new ribbon disk with a max-tb unlink ribbon derivative; see the right side of \cref{fig:dec-derivative}.

\begin{figure}[ht]
	\centering
	\begin{overpic}[scale=.37]{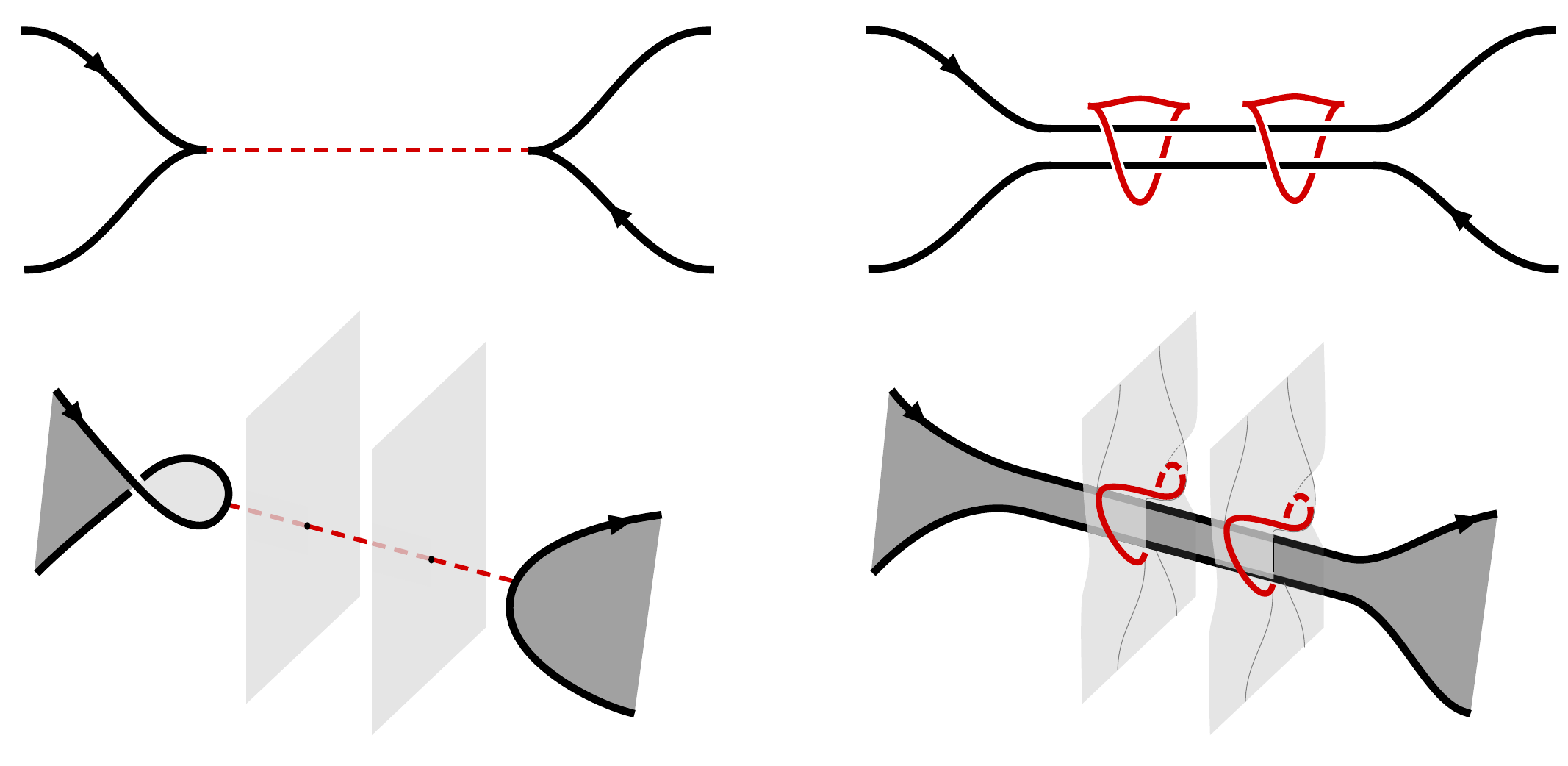}
       
	\end{overpic}
	\caption{The top row is an honest front projection, with the ribbon disk omitted. The bottom row is a tilted view with the ribbon disk shaded.}
	\label{fig:dec-derivative}
\end{figure}

Note that if $\Lambda$ is \emph{strongly decomposably slice} \cite{breen2024regularlysliceimpliesoncestably}, i.e.\ $\Lambda$ admits a \emph{Legendrian handle graph presentation} from a max-tb unlink \cite{sabloff2024upper}, i.e.\ $\Lambda$ admits a decomposable slice disk presentation where all $1$-handles commute, the proof is more straightforward, as one can explicitly identify the entire ribbon disk and ribbon derivative in one fell swoop. The reason why the above proof is phrased inductively is because some of the $1$-handles may be nested and non-commutable; see the above references for more details.
\end{proof}

\section{Transverse \(R\)-links and transverse derivative links}\label{sec:transverse}

In this section we consider $R$-links and derivative links in the transverse setting. 

\begin{lemma}\label{lemma:qp_unlink_derivative}
Let $\mathcal{T}\subset (S^3, \xi_{\mathrm{st}})$ be a symplectically slice transverse knot. Then $\mathcal{T}$ admits a Seifert surface supporting a max-sl unlink derivative.      
\end{lemma}

\begin{proof}
As $\mathcal{T}$ is symplectically slice, it is quasipositive \cite{boileau2001quasipositive,rudolph1983seifertribbons}. Specifically, it is transversely isotopic to the closure of a quasipositive braid $\beta\in B_n$ in the standard radial contact structure. Let $\beta = b_1b_2\cdots b_m$ denote the factorization into quasipositive band generators (i.e.\ arbitrary conjugates of the standard Artin generators). This means that each $b_i$ is a band, possibly intersecting the braid axis disks, with a single positive crossing. 

Let $D_{\beta}$ denote the immersed Bennequin surface obtained by attaching the bands $b_1, \dots, b_m$ to the braid axis disks $D_1 \cup \cdots \cup D_n$. Since $\mathcal{T}$ is quasipositive and slice, $D_{\beta}$ is a ribbon disk. See \cref{fig:qpbraid}. 

\begin{figure}[ht]
	\centering
	\begin{overpic}[scale=.35]{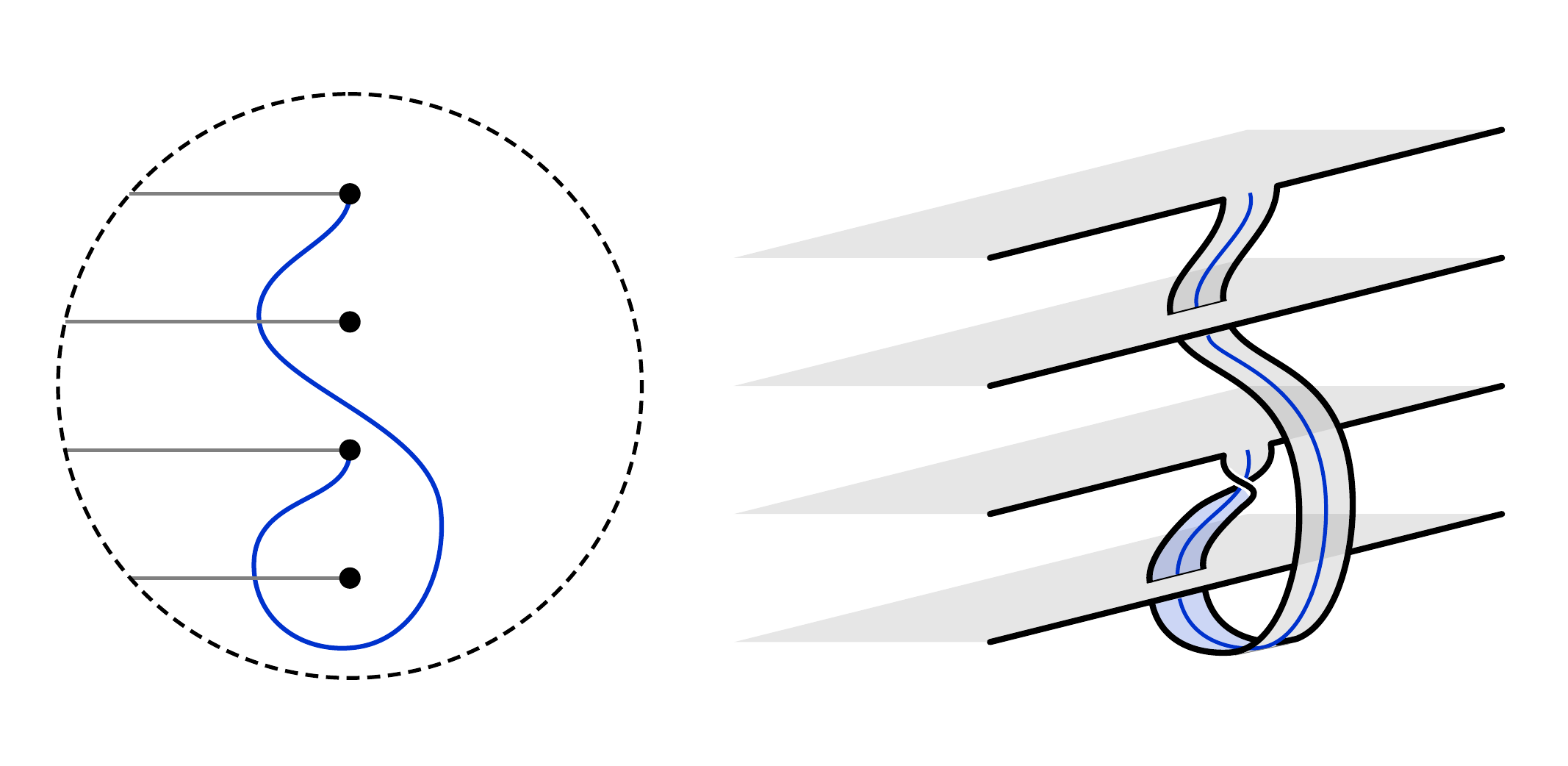}
       
	\end{overpic}
	\caption{A quasipositive braid generator viewed on the left as an element of the mapping class group of a punctured disk (the right-handed twist about the blue arc), and as a band attached to the braid axis disks on the right.}
	\label{fig:qpbraid}
\end{figure}

Up to transverse isotopy, we may assume that each ribbon intersection of $D_{\beta}$ is a radial transverse arc in the interior of the braid disks. (Diagrammatically, this corresponds to arranging for the positive crossing of the band to be close to one of the feet of the band.) Moreover, also after a transverse isotopy, we can assume that the radial coordinates of all ribbon intersections are pairwise distinct. 

Next we describe the derivative link. Assume that the radius of each braid axis disk is $R > 0$. For each of these disks $D_j$, let $0 < r_{1,j} < \cdots < r_{k_j,j} < R$ denote the radial coordinates of the $k_j$-many ribbon intersections supported in $D_j$. Choose radial coordinates $0 < r_{1,j}^* < \cdots < r_{k_j,j}^* < R$ so that $r_{\ell,j} < r_{\ell,j}^* < r_{\ell+1,j}$ for $\ell = 1, \dots, k_j-1$ and $r_{k_j,j} < r_{k_j,j}^*$. Let $T_{\ell,j}$ denote the round circle in $D_j$ of radius $r_{\ell,j}^*$ centered at the origin and let $T = \bigcup_{j=1}^n \bigcup_{\ell=1}^{k_j} T_{\ell,j}$. Note that each $T_{\ell,j}$ is a max-sl transverse unknot, being a radial circle in the standard radial contact structure. Moreover, $T$ is smoothly an unlink. Hence, $T$ is a max-sl unlink. 

Finally, let $\Sigma_{\beta}$ denote the embedded Seifert surface obtained by performing a cellar-door resolution on each ribbon singularity of $D_{\beta}$. We may perform the resolution in sufficiently small neighborhoods supported away from $T$. By construction, $T\subset \Sigma_{\beta}$ is a derivative link. This completes the proof. 
\end{proof}

\begin{lemma}\label{lemma:tranRlink_sympl_slice}
If a transverse knot is a component of a transverse $R$-link, it is symplectically slice. 
\end{lemma}

\begin{proof}
By assumption, inadmissible transverse $0$-surgery along the $R$-link produces $(\#^g (S^1 \times S^2), \xi_{\mathrm{st}})$. This contact manifold is (Weinstein) filled by $(\natural^g (S^1 \times B^3), \omega_{\mathrm{st}})$. 

Thus, combining \cref{lemma:admissible} and \cref{thm:admissible_handle}, we may attach symplectic $2$-handles with symplectic co-cores to $\natural^g(S^1 \times B^3)$ along the dual link to induce a weak symplectic filling of the boundary. Since the boundary is diffeomorphic to $S^3$ and admits a weak symplectic filling, the boundary is contactomorphic to $(S^3, \xi_{\mathrm{st}})$. Moreover, the original transverse $R$-link bounds the symplectic co-core disks of the $2$-handles. Each component is therefore symplectically slice in the weak filling. 

Because the weak filling of $(S^3, \xi_{\mathrm{st}})$ is constructed by attaching $2$-handles to $\natural^g(S^1 \times B^3)$, it is a homotopy ball and hence a minimal symplectic filling. The unique minimal symplectic filling of $(S^3, \xi_{\mathrm{st}})$ is $(B^4, \omega_{\mathrm{st}})$ \cite{gromov1985pseudo}.
\end{proof}

Next we prove \cref{thm:tranKauffman}.

\begin{proof}
By \cref{lemma:qp_unlink_derivative}, \eqref{part:TK1} implies \eqref{part:TK3}.
Since max-sl transverse unlinks are transverse $R$-links, \eqref{part:TK3} implies \eqref{part:TK2}. Thus, it remains to show that \eqref{part:TK2} implies \eqref{part:TK1}. 

Suppose $\mathcal{T}\subset (S^3, \xi_{\mathrm{st}})$ is a transverse knot with $\mathrm{sl}(\mathcal{T}) = -1$ and $\Sigma\subset S^3$ is a Seifert surface with a transverse $R$-link derivative $T:=T_1 \cup \cdots \cup T_g$. We may assume that the characteristic foliation of $\Sigma$ is Morse-Smale, so that 
\[
-1 = \mathrm{sl}(\mathcal{T}) = -(e_+ - h_+) + (e_- - h_-)
\]
where $e_{\pm}, h_{\pm}$ are the counts of positive (resp.\ negative) elliptic and hyperbolic singular points, respectively. 

In the $0$-surgery $S^3_T(0)$, the image of $\mathcal{T}$ is a topological unknot $U$ (since $\Sigma $ surgers along $T$ to yield a disk). We claim that $
\mathrm{sl}(U) = -1$. Note that each component of the derivative $T$ contributes to two boundary components of the cut-open Seifert surface $\Sigma':=\Sigma - \nu(T)$ in $S^3_T(0)$. When we cap off these two boundary components in $S^3_T(0)$ with standard meridional disks in the transverse surgery torus, the orientation forces one capping disk to contribute a positive elliptic point and the other to contribute a negative elliptic point. Thus, 
\[
\mathrm{sl}(U) = -(e_+ + g - h_+) + (e_- + g - h_-) = -(e_+ - h_+) + (e_- - h_-) = -1.
\]
Since $U$ is a max-sl unknot in $(\#^g S^1\times S^2, \xi_{\mathrm{st}})$, performing a further inadmissible transverse $0$ surgery along $U$ produces $(\#^{g+1} S^1\times S^2, \xi_{\mathrm{st}})$. It follows that $\mathcal{T}\cup T$ is a transverse $R$-link, which by \cref{lemma:tranRlink_sympl_slice} implies that $\mathcal{T}$ is symplectically slice. 
\end{proof}

\bibliography{references}
\bibliographystyle{amsalpha}
\end{document}